\documentclass[12pt,a4 paper]{article}
\usepackage{amssymb}
\usepackage{amsmath}
\usepackage{amsthm}
\usepackage{amsfonts}
\usepackage{amstext}
\usepackage{cite}
\newtheorem{Theorem}{Theorem}[section]
\newtheorem{Lemma}[Theorem]{Lemma}
\newtheorem{Proposition}[Theorem]{Proposition}

\newtheorem{Remark}[Theorem]{Remark}
\def\o{\Omega}
\def\om{\omega}

\def\R{\mathbb{R}}

\def\e{\varepsilon}
\def\rr{R^\nu}
\def\P{\mathbb{P}}
\newcommand{\dif}{\mathrm{d}}
\newcommand{\di}{\mathrm{div}}
\newcommand{\p}{\partial}
\def\C{\mathrm{curl}}
\begin{document}
\title{ Vanishing Viscous Limits for 3D Navier-Stokes Equations with A Navier-Slip Boundary Condition\footnote{This research is partially supported by Zheng Ge Ru Funds, Hong Kong RGC Earmarked Research Grant CUHK4042/08P, CUHK4040/06P, and CUHK4041/11P and by Scientific research
plan projects of Shaanxi Education(09JK770); China Postdoctoral
Science Foundation(20090461305); by Doctor Funds of Yichun Univeristy.}}
\author{ Lizhen Wang$^{a,1}$, Zhouping Xin$^{b,2}$, Aibin Zang$^{b,c,3}$}
\date{}
\maketitle

\begin{center}
$^a$ The Department of Mathematics Northwest University, Xi'an, Shaanxi,  P.R. China\\

$^b$ The Institute of Mathematical Sciences, The Chinese University
of Hong Kong, Shatin, NT, Hong Kong\\
$^c$ The School of Mathematics and Computer Science, Yichun University, Yichun, Jiangxi, P.R.China\\[3mm]
\footnotesize{$^1$Email: lwang19@tom.com, $^2$Email: zpxin@ims.cuhk.edu.hk, $^3$Email: zangab05@126.com}

\end{center}

\vskip 5mm

\begin{abstract}
In this paper, we investigate the vanishing viscosity limit for
solutions to the Navier-Stokes equations with a Navier slip boundary
condition on general compact and smooth domains in $\mathbf{R}^3$.
We first obtain the higher order regularity estimates for the
solutions to Prandtl's equation boundary layers. Furthermore, we
prove that the strong solution to Navier-Stokes equations converges
to the Eulerian one in $C([0,T];H^1(\Omega))$ and
$L^\infty((0,T)\times\o)$, where $T$ is independent of the
viscosity, provided that initial velocity is regular enough.
Furthermore, rates of convergence are obtained also.


\end{abstract}
\numberwithin{equation}{section}

\numberwithin{equation}{section}

\section{Introduction}

\setcounter{equation}{0} ~~~~~~ In this paper, we consider the
vanishing viscosity limit problem from the Navier-Stokes flows on a
general 3-dimensional bounded domain with Navier-slip boundary
condition. The viscous flow is governed by
\begin{equation}\label{1.1}
\left\{\begin{array}{cl} \partial_t u^\nu-\nu\Delta u^\nu+( u^\nu\cdot\nabla) u^\nu+\nabla \pi^\nu=0, &\mbox{in}~~\Omega\times (0, T), \\[3mm]
\nabla\cdot u^\nu = 0, &\mbox{in}~~\Omega\times (0, T), \\[2mm]
\end{array}\right.
\end{equation}
with the boundary conditions
\begin{equation}\label{NAB}
u^\nu\cdot \vec{n}=0, \quad (\C\, u^\nu)\times\vec{n}=0, \quad
\mbox{on}~~\p\Omega\times(0, T),
\end{equation}
where $\vec{n}$ is the outnormal of $\p\o$, and initial velocity $$u^\nu|_{t=0}=u_0(x), ~~~\mbox{in}~~\Omega.$$
Here the unknowns are the velocity $u^\nu(t, x)$ and the scalar
pressure $\pi^\nu(t, x)$,  $u_0(x)$ is the given initial velocity and
 the corresponding problem for the Euler equations reads
\begin{equation}\label{1.3}
\left\{\begin{array}{cl} \partial_t u^0+( u^0\cdot\nabla) u^0+\nabla \pi^0=0, &\mbox{in}~~\Omega\times (0, T), \\[3mm]
\nabla\cdot u^0 = 0, &\mbox{in}~~\Omega\times (0, T), \\[2mm]
u^0\cdot \vec{n}=0,&\mbox{on}~~\p\Omega\times(0, T),\\[2mm]
u^0|_{t=0}=u_0(x), &\mbox{in}~~\Omega.
\end{array}\right.
\end{equation}

It should be noted that the slip boundary condition (1.2) is a special case of the more general Navier-slip boundary condition
\begin{equation}\tag{1.4}
u^\nu\cdot\vec{n}=0, \quad (D(u^\nu)\vec{n} + \alpha u^\nu)_\tau=0,
\quad {\rm on}~~\, \p\o\times(0,T),
\end{equation}
where $D(u^\nu)=\frac{1}{2}(\nabla u^\nu + (\nabla u^\nu)^t)$ and $\tau$ is any tangent direction on $\p\o$.

The problem of vanishing viscosity limits for the Navier-Stokes
equations is a classical issue. In the absence of physical
boundaries, then any smooth solutions to the Euler system can be
approximated by the ones to Navier-Stokes equations, see
\cite{BV, CF,CW,EM,K1,K2,SW,MA}. However, in the presence of physical
boundaries, this problem is a challenging problem due to the
possible formation of boundary layers. The problem for the non-slip
boundary condition was formally derived by Prandtl in \cite{PR}, in
which it was obtained that the boundary layer can be described by an
initial-boundary problem  for a nonlinear degenerate
parabolic-elliptic couple system called the Prandtl's equations.
Under monotonic assumptions on the velocity of the outflow, Oleinik
and her collaborators established the local existence of smooth
solutions for boundary value problem of the 2-dimensional Prandtl's
equations \cite{OS}. In this case, the existence and uniqueness of
global solutions to the Prandtl's equations was established by Xin,
Zhang \cite{XZ}(also see \cite{XY}). In \cite{SC}, Sammartino and
Caflisch obtained the local existence of the analytic solutions to
the Prandtl's equations and a rigorous theory  on the boundary layer
in incompressible fluids with analytic data in the frame of the
abstract Cauchy-Kowaleskaya theory.

However, the usual non-slip assumption was not always accepted from
experimental facts. In \cite{NA}, Navier first proposed the slip
boundary condition (1.4) i.e. the tangential velocity proportional
to the tangential component of the viscous stress while maintaining
the no-flow condition in the normal direction, which is now called
Navier boundary condition. This boundary condition was rigorously
justified as the effective boundary conditions for flows over rough
boundaries, see \cite{JM}.

In contrast to the case of non-slip boundary condition, one would
expect that the boundary layers are much weaker for the Navier-slip
boundary condition, (1.2), and thus it should be easier to settle
the problem of vanishing viscosity. Indeed, there have been many
interesting studies along this line. For 2-dimensional smooth
domains, Yudovich \cite{Y} and Lions, P.L. \cite{LION} studied this
problem for a special class of Navier-slip conditions, the vorticity
free condition, for the incompressible Navier-Stokes equations, i.e.
(1.4) in 2-dimensional space. For the general Navier-slip
conditions, Clopeau, et.al. \cite{CMR}, Lopes Filho, Nussenveig
Lopes and Planas \cite{LNP} obtained that the solution $u^\nu$ to
(\ref{1.1}) converges to the solution $u^0$ of Euler equations in
$L^\infty(0,T; L^2(\R^2_+))$ assuming that initial vorticity is
uniformly bounded.  More generally, Iftimie and Planas \cite{IP}
observed that in both dimension two and three a direct $L^2$
estimate yields the strong $L^2$ convergence to Euler equation, and
that the convergence in $H^2$ is impossible in general. Thus, higher
order (weaker) boundary layers must appear in general. This was
investigated further by Iftimie and Sueur in \cite{IS} for
Navier-slip conditions (1.4) with fixed slip length
($\alpha=const.$) and they improved the strong $L^2$ convergence
with the rate $O(\nu^\frac{3}{4}).$ Furthermore, Wang, X.P., Wang,
Y.G. and Xin, Z.P. \cite{WWX} studied the asymptotic behavior of
solutions to (\ref{1.1}) with Navier boundary conditions (1.4) for
variable slip length ($\alpha=\nu^\gamma$). They observed that the
vanishing viscosity limit for the problem (\ref{1.1}) with boundary
conditions (\ref{NAB}) for $\alpha=\nu^\gamma$ should be influenced
by the amplitude of the slip length. It should be noted that the
approach in \cite{WWX} can yield easily the leading profile
expansion of boundary layers not only in $L^\infty(0,T;L^2(\Omega))$
as given in \cite{IS}, but even in $L^\infty((0,T)\times\Omega)$.
More recently, Masmoudi and Rousset \cite{Mas} proved that the
solutions to (\ref{1.1}) converge uniformly to the one of the Euler
equations in the spatial and time variables under the frame of
conormal Sobolev space.  However,  for general Navier  boundary
conditions, it is difficult to obtain the convergence in higher
order, even in $H^1$ as mentioned in \cite{IP}.

In 2007, Xiao and Xin \cite{xin} first studied the problem
(\ref{1.1}) with the completely slip boundary conditions, i.e. (1.2)
which is a special case of (1.4). For the special case of flat
boundaries, Xiao and Xin obtained the uniform $H^2$-convergence
theory with the optimal convergence rate.  Later,
Beir\~{a}o da Veiga and Crispo obtained the corresponding
$L^p$-theory and the $W^{k,p}$-convergence in \cite{BVC,BVC2}, and
they pointed out that in general it is impossible to have the
$H^2$-convergence for general 3D domains \cite{BVC3}, see also
\cite{XX}. It is also proposed in \cite{BVC} as an challenging open
problem to study the uniform $H^1$-convergence theory of solutions
to the problem (\ref{1.1})-(1.2) for general 3D domains. The only
previous results for this problem is due to Xiao-Xin in \cite{XX}
where they obtained the convergence in $L^\infty((0,T)\times
H^1(\Omega))$ for general smooth 3D domains with the rate $O(\nu)$
with the complete slip boundary condition (1.2) under the stringent
additional condition that the initial vorticity vanishes on the
boundary of the domain.

The main purpose of this paper is to establish the
$L^\infty((0,T)\times H^1(\Omega))$ convergence theory for the
solutions to the Navier-stokes system (1.1) with the slip boundary
condition (1.2) in general 3D domains to the solution to the
inviscid problem (1.3) with a rate $O(\nu^{\frac{1}{4}})$ and to
prove the optimal rate of convergence in
$L^\infty((0,T)\times\Omega)$.

Some of the main difficulties involved with general domains can be illustrated as follows. As pointed out in [12, 30], the solution, $u^\nu$, to (1.1)-(1.3) is expected to have the form
\begin{equation}
u^\nu(t,x)=u^0(t,x)+\sqrt{\nu}u^b(t,x,\frac{z}{\sqrt{\nu}})+O(\nu),
\end{equation}
where $z$ is equivalent to the distance between $x$ and the
boundary, $u^b$ is the leading order boundary layer which is smooth
and decreasing fast in the last variable. For flat boundaries, $u^b$
vanishes identically and thus it is possible to obtain the uniform
$H^3$ or $W^{2,p}$ ($p>3$) convergence theory as in [31, 3, 4].
While for general curved domains, due to the curvature of
$\partial\Omega$, $u^b$ does not vanish in general, which leads to
new difficulties in the estimates of derivatives by the methods in
[12, 31, 3, 4, 30, 32]. Our main strategy to overcome these
difficulties is outlined as follows. The first key step is to
estimate the leading order boundary layer profile $u^b$ to gain
higher order regularities than those known ones in [12, 30], see
(3.9), in particular, the uniform $W^{k,p}$-estimate in slow spatial
variable and $H^s$-estimates for time and fast spatial variables.
The second is to establish the uniform $L^p$-bound ($3<p\leq 6$) for
the remainder in the asymptotic expansion of $u^\nu$ (see (3.1)). It
should be noted that the $L^\infty(0,T;L^p(\Omega))$ convergence
theory with a uniform rate follows easily from the uniform
$L^\infty$-bounds on derivatives obtained in [18] and the
$L^\infty(0,T;L^2(\Omega))$ convergence theory in [12, 30]. However,
our approach yields a better rate of convergence for $p>2$, which
will be useful for our main results. The next step is to derive the
$H^1$-bound in the order $O(\nu^{-\frac{1}{2}})$ for the remainder
in the asymptotic expansion (3.1). This follows from the important
results in [18] that the gradients of the solution, $u^\nu$, to
(1.1)-(1.2) are uniformly bounded for sufficient smooth initial
velocity and the uniform $L^p$-estimates derived in the previous
step. Then the desired $C([0,T]; H^1(\Omega))$ convergence theory
with a rate $O(\nu^{\frac{1}{4}})$ follows from this and the
asymptotic ansatz provided that the initial velocity is regular
enough. The final step is to estimate the $W^{1,p}$ ($p>3$) bound
for the remainder in the asymptotic expansion by following the
arguments for $H^1$-estimates and to obtain the
$L^\infty((0,T)\times\Omega)$ convergence with an optimal rate of
order $O(\nu^{\frac{1}{2}})$.

The rest of the paper is organized as follows. First, in section 2,
we state some notations and preliminary results to be used later.
Then the asymptotic ansatz of the solution $u^\nu$ to (1.1)-(1.2)
and main convergence results are given in section 3. The desired
higher order regularities of the leading order boundary layer
profiles are obtained in section 4. In section 5, the uniform
$L^p$-bound for the remainder of the asymptotic ansatz is derived
for $3<p\leq 6$. Then we derive the $H^1$-estimate of the remainder
and prove the $C([0,T]; H^1(\o))$ convergence of $u^\nu$ to the
solution, $u^0$, to (1.3) with a rate $O(\nu^\frac{1}{4})$, in
section 6. Finally, we also obtain the $W^{1,p}$-estimates with
$3<p\leq 6$ and prove the convergence in $L^\infty((0,T)\times
\Omega)$ with an optimal rate of order $O(\nu^{\frac{1}{2}})$ in
section 6.

\section{Notations and preliminaries}
\setcounter{equation}{0} ~~~~~~In this section, we will give some
notations and preliminary results which will be employed later. Let
$\Omega$ be a bounded smooth domain in $\R^3$, $k\in\R$, and
$1<p\leq+\infty$.

In the following sections, we will utilize the classical Lebesgue
spaces $(L^p(\Omega),\|\cdot\|_p)$,
$(L^p(\partial\Omega),\|\cdot\|_{p,\partial\Omega})$ and the
standard Sobolev spaces $(W^{k,p}(\Omega),\|\cdot\|_{k,p})$ and
trace spaces
$(W^{m,p}(\partial\Omega),\|\cdot\|_{m,p,\partial\Omega})$.
$L^2_\sigma(\Omega)$ is the subspace of $L^2(\Omega)$ satisfying the
divergence free condition.

Let $(W^{k,m,l,p}(\Omega),\|\cdot\|_{k,m,l,p})$ be the anisotropic
Sobolev spaces defined as follows: for $k,m,l\in \mathbf{N},~\, p\ge 1$
\begin{eqnarray*}
&W^{k,m,l,p}(\Omega\times\mathbf{R}_+)=\left\{g(x,z)\in L^p(\Omega\times R_+): (1+z^{2k})^{\frac{1}{p}}\partial^\alpha_x\partial^\beta_zg(x,z)\in
L^p(\Omega\times R_+)\right.,\\
&\left. |\alpha|\leq m, \beta\in\mathbf{N}, \beta\leq l\right\}
\end{eqnarray*}
with the norm
$$\|g\|^p_{k,m,l,p}=\sum_{|\alpha|\le m,|\beta|\le l}\iint_{\Omega\times\mathbf{R}_+}(1+z^{2k})|\partial^\alpha_x\partial^\beta_zg(x,z)|^{p}\dif x\dif z.$$
 When $p=2$, for simplicity, denote
$W^{k,m,l,2}(\Omega)$ by $H^{k,m,l}(\Omega)$.
In the rest of paper, $C$ will be a generic
constant, which may change from line to line, but is independent of the viscosity.

In the sequel, we list several results including existence theorems of solution to the Navier-Stokes equations and the Euler equations, which will be used later.
\begin{Lemma}
Let $\Omega$ be given a smooth and open set and $1<p<\infty$. Then
the following inequality holds true: There exists $C>0$, such that
$$\|u\|_{p,\partial\Omega}\le C\|u\|^{1-\frac{1}{p}}_p\|u\|^{\frac{1}{p}}_{1,p},~\, \forall u\in W^{1,p}(\o).$$
\end{Lemma}
\begin{proof}
This follows easily from Lemma 7.44 in\cite{ad}.\end{proof}
\begin{Lemma}
Let $u\in W^{s,p}(\o)$ be a vector-valued function. Then for $s\ge
1$
$$\|u\|_{s,p}\le C\left(\|\nabla\times u\|_{s-1,p}+\|\di u\|_{s-1,p}+\|u\cdot n\|_{s-\frac{1}{p},\p\o}+\|u\|_{s-1,p}\right).$$
\end{Lemma}
\begin{proof}
See \cite{BB,xin}.
\end{proof}
\begin{Lemma}
For all $u\in W^{1,p}(\o),~\, 1<p<+\infty,$ there exists $C>0$ such that
$$\|u\|_p\le C\|\nabla u\|_p,$$
$$\|\nabla u\|_p\le C(\|\nabla\times u\|_p+\|\di u\|_{p}),$$ for all $u $ such that $ u\cdot \vec{n}|_{\p\o}=0, ~\mbox{or}~u\times \vec{n}|_{\p\o}=0.$
\end{Lemma}
\begin{proof}
This is proved in \cite{xin}.
\end{proof}

It should be noted that in Lemma 2.3 and what follows, $\vec{n}$ denotes the unit normal of $\partial\Omega$.

\begin{Lemma}
 Let $u$ be a smooth function such that
$\C \,u\times\vec{n}=0$. Then $\omega=\mathrm{curl}\,u$ satisfies the
following equality on $\partial\Omega$
$$ -\frac{\p\omega}{\partial
\vec{n}}\omega=(\epsilon_{1jk}\epsilon_{1i\gamma}+\epsilon_{2jk}\epsilon_{2i\gamma}+\epsilon_{3jk}\epsilon_{3i\gamma})\omega_j\omega_i\partial_kn_\gamma, $$
where $\epsilon_{ijk}$ denotes the totally anti-symmetric tensor
such that $(\varphi\times\psi)_i=\epsilon_{ijk}\varphi_j\psi_k$.
\end{Lemma}
\begin{proof}
This follows from \cite{BVB}.
\end{proof}
\begin{Lemma}(Hardy's Inequality) If $\o\subset \mathbb{R}^n, n\ge2$, is a bounded lipschitz domain, then
$$\int_\o \frac{|u(x)|^p}{d(x,\p\o)^{p-\beta}}\dif x\le C\int_{\o}\frac{|\nabla u|^p}{d(x,\p\o)^{-\beta}}\dif x,~\, \forall u\in C_0^\infty(\Omega),$$
for all $\beta<p-1$, where $d(x,\p\o)$ is the distance between $x$ and $\p\o.$
\end{Lemma}
\begin{proof}
See \cite{N}
\end{proof}
\begin{Lemma}(Gronwall's Lemma)

a) (Differential Version) Suppose that $h$ and $r$ are integrable on $(a,b)$ and nonnegative a.e. in $(a,b)$. Further assume that  $y\in C([a,b])$, $y'\in L^1(a,b)$, and
$$y'(t)\le h(t)+r(t)y(t) ~\mbox{for a.e.}~ t\in(a,b).$$
Then
$$y(t)\le \left[y(a)+\int^t_a h(s)\exp\left(-\int^s_a r(\tau)\dif\tau\right)\dif s\right]\exp\left(\int^t_a r(s)\dif s\right),~\, t\in[a,b].$$

b) (Integral Form) Suppose that $h$ is continuous on $[a,b]$, $r$ is
nonnegative and integrable on $(a,b)$ and $y\in C([a,b])$  satisfies
the following inequality:
$$y(t)\le h(t)+\int_a^t r(s)y(s)\dif s ~\mbox{for a.e.}~ t\in(a,b).$$
Then
$$y(t)\le h(t)+\int^t_a h(s)r(s)\exp\left(\int^t_s r(\tau)\dif\tau\right)\dif s, ~\,t\in[a,b].$$

c) (Local Version) Let $T,\alpha, c_0>0$ be given constants and $h$ be a nonnegative integrable function on $[0,T]$. Assume that $y\geq 0$ and $y\in C^1([0,T])$ satisfies
$$y'(t)\le h(t)+c_0 y(t)^{1+\alpha} ~\mbox{a.e.}~ t\in(0,T).$$
Let $t_0\in [0,T]$ be such that $\alpha c_0 H(t_0)^\alpha t_0<1$, where
$$H(t)=y(0)+\int^t_0h(s)\dif s.$$
Then for all $t\in[0,t_0]$ there holds
$$y(t)\le H(t)+H(t)\left((1-\alpha c_0H(t)^\alpha t)^{-\frac{1}{\alpha}}-1\right).$$
\end{Lemma}
\begin{proof}
a) and b) are standard, see \cite{NS}. For the proof of part c), one can refer to the Appendix of \cite{dien}.
\end{proof}

The following theorems ensures the existence of strong solutions to the Euler equations.
\begin{Theorem} Assume that $\Omega$ is a regular bounded open set
of $\mathbf{R}^3$. Let $m$ and $p$ be given, $p\geq1$,
$m>1+\frac{3}{p}$. Then for each $u_0\in W^{m,p}(\Omega)$ such that
$\mathrm{div} u_0=0$ and $u_0\cdot \vec{n}=0 $ on $\partial\Omega$, there exist $\bar{T}\le T$ and
a unique function $u^0$ and a function $\pi$ on $(0,\bar{T})$ , such that $u^0\in
C([0,\bar{T}];W^{m,p}(\Omega))\cap C^1([0,\bar{T}];W^{m-1,p}(\Omega))$ and $\pi^0\in
L^\infty(0,\bar{T};W^{m+1,p}(\Omega))$ solve (\ref{1.3}).
\end{Theorem}
\begin{proof}
See \cite{Temam} or \cite{BB}.
\end{proof}

 Let $(Z_j)_{j=1\cdots N}$ be a set of generators of vector fields tangential to $\p\o$. For a multiindex $\beta$, $Z^\beta=Z_1^{\beta_1}\cdots Z_N^{\beta_N}$,  define
$$H_{co}^m(\o)=\{f\in L^2(\o): Z^\beta f\in L^2(\o) \,\,\mbox {for all }\,\, |\beta|\le m\}$$
with $$\|f\|^2_{H_{co}^m(\o)}=\sum_{|\beta|\le m}\|Z^\beta
f\|_{L^2(\o)}^2.$$ Similarly, one can define the space
$W^{m,\infty}_{co}$ and $f\in W^{m,\infty}_{co}$ if
$$\|f\|_{W^{m,\infty}_{co}(\o)}=\sum_{|\beta|\le m}\|Z^\beta f\|_{L^\infty(\o)}<\infty.$$
One can also define the space  $E^m$ by
$$E^m=\{u\in H_{co}^m(\o)|\nabla u\in H_{co}^{m-1}(\o) \}$$
with the obvious norm. The following important uniform
well-posedness results for solutions to (\ref{1.1}-\ref{NAB}) in the
above spaces are due to Masmoudi and Rousset [18].

\begin{Theorem} Let $m$ be an integer satisfying $m>6$ and $\Omega$ be a $C^{m+2}$ domain. Consider $u_0\in E^m\cap L_\sigma^2(\o)$ such that $\nabla u_0\in W^{1,\infty}_{co}(\o)$. Then there exists a positive constant $T$ such that for all sufficiently small $\nu$ there exists a unique solution, $u^\nu\in C([0,T]; E^m)$ to (\ref{1.1}, \ref{NAB}) such that $\|\nabla u^\nu\|_{1,\infty}$ is bounded on $[0,T]$. Moreover, there exists $C$ independent of $\nu$,  such that
$$\sup_{t\in [0,T]}(\|u^\nu(t)\|_{H_{co}^m(\o)}+\|\nabla u^\nu(t)\|_{H_{co}^{m-1}(\o)}+\|\nabla u^\nu(t)\|_{W^{1,\infty}_{co}(\o)})\le C.$$
\end{Theorem}
From now on, the time $T$ is taken to be finite and fixed unless stated otherwise.

\section{Asymptotic expansions of the  solution and main result}

 \setcounter{equation}{0}
~~~~~~ In this section, we are going to study the asymptotic
expansions of the strong solution as in Theorem 2.8 and state the
main results in this paper.

First, choose a smooth function $\varphi\in
C^{\infty}(\mathbb{R}^3;\mathbb{R})$ such that  in a neighborhood
$\Lambda$ of $\p\o$, one has that
$\o\cap\Lambda=\{\varphi>0\}\cap\Lambda,~\,\o^c\cap\Lambda=\{\varphi<0\}\cap\Lambda,~\,\p\o\cap\Lambda=\{\varphi=0\}\cap\Lambda$
and it is normalized such that $|\nabla\varphi|=1$ for all
$x\in\Lambda.$  Thus $\varphi$ is regarded as a distance between $x$
and $\p\o$ for $x\in\Lambda$ without restriction. It is assumed that
$\Lambda=\{x\in\o:\varphi(x)<\eta\}$ for a small number $\eta>0.$ We
define a smooth extension of the normal unit vector $\vec{n}$ inside
$\Omega$ by taking $\vec{n}=\nabla\varphi.$

 As in \cite{WWX,IS}, the solution $u^\nu$ to (1.1)-(1.2) is expected to be described by the following ansatz:
 \begin{eqnarray}
 u^\nu(t,x)=u^0(t,x)+\sqrt{\nu} u^b(t,x,\frac{\varphi(x)}{\sqrt{\nu}})+\nu v(t,x,\frac{\varphi(x)}{\sqrt{\nu}})+\nu R^\nu(t,x);\label{3.1}\\
 \pi^\nu(t,x)=\pi^0(t,x)+\sqrt{\nu} p (t,x,\frac{\varphi(x)}{\sqrt{\nu}})+\nu q(t,x,\frac{\varphi(x)}{\sqrt{\nu}})+\nu \kappa(t,x)\label{3.2}.
 \end{eqnarray}
Plugging (3.1) and (3.2) into (1.1) leads to
\begin{eqnarray}
&&\p_z u^b\cdot \vec{n}=0; \label{3.3}\\
&&\di_x u^b=-\p_z v\cdot \vec{n}. \label{3.4}
\end{eqnarray}
Following the argument in \cite {WWX} shows easily that $p\equiv0.$ On the other hand,  the  term of the order $O(\sqrt{\nu})$ is
\begin{equation}\label{3.5}\tag{3.5}
\p_t u^b-\p_z^2 u^b+\frac{u^0\cdot \vec{n}}{\varphi(x)}z\p_z
u^b+u^0\cdot\nabla u^b+u^b\cdot\nabla u^0+u^b\cdot \vec{n}\p_z
u^b+\vec{n}\p_z q
\end{equation}

Modifying slightly the proof in \cite{IS}, one can prove that if $u^b(t,x,0)\cdot \vec{n}(x)=0$ and $u^b$ solves the following equations
\begin{equation}\label{3.6}\tag{3.6}
\begin{aligned}
\p_t u^b-\p_z^2 u^b+\frac{u^0\cdot \vec{n}}{\varphi(x)}z\p_z
u^b+(u^0\cdot\nabla u^b+u^b\cdot\nabla u^0)\times \vec{n}=0,
\end{aligned}
\end{equation}
then $u^b\cdot \vec{n}= 0$, for all $(t,x,z)\in (0,T)\times \Omega\times\mathbb{R}_+.$

Therefore, we can infer that $u^b$ satisfies the following system
\begin{equation}\label{3.7}\tag{3.7}
\left\{\begin{aligned}
&\p_t u^b-\p_z^2 u^b+\frac{u^0\cdot \vec{n}}{\varphi(x)}z\p_z u^b+(u^0\cdot\nabla u^b+u^b\cdot\nabla u^0)\times \vec{n}=0,\\
&(u^0\cdot\nabla u^b+u^b\cdot\nabla u^0)\cdot \vec{n}=\p_z q.
\end{aligned}\right.
\end{equation}
with boundary and initial conditions
\begin{equation}\label{3.8}\tag{3.8}
\left\{\begin{aligned}
&u^b\cdot \vec{n}=0,~\,\p_z u^b =-\C\,u^0\times\vec{n}, ~\, \mbox{on}~\, z=0,\\
& u^b(0,x,z)=0.
\end{aligned}\right.
\end{equation}
The following proposition follows essentially as in [12],
\begin{Proposition}
There exists a unique pair $(u^b,q)$ which solves (3.7)-(3.8) with the following
$$u^b\in L^\infty(0,T; H^{k,2,0})\cap L^2(0,T;H^{k,2,1})$$ for all $k\in\mathbb{N}$ and $\p_z u^b\in L^\infty((0,T)\times\Omega\times\mathbb{R}_+).$

Moreover, $u^b$ vanishes for $x$ outside the neighborhood $\Lambda$ and $u^b\cdot \vec{n}=0$ for all $(t,x,z)\in (0,T)\times \Omega\times\mathbb{R}_+.$ Consequently, it holds that
\begin{equation*}
\sup_{t\in [0,T]}\|u^\nu-u^0\|_{L^2(\Omega)}\le C\nu^{\frac{3}{4}},
\end{equation*}provided  that initial velocity $u_0\in H^3(\o).$
\end{Proposition}

In this paper, we aim to improve the regularity of $u^b$ which is
the solution of problem (3.7) with (3.8) under the assumption of
higher order regularities for the initial velocity and to obtain our
main results. First, we have
\begin{Theorem}
Let $u_0\in H^s$ for $s\ge 6$, be a divergence free vector field
satisfying the boundary conditions (\ref{NAB}). Suppose $\Omega$ to
be a $C^{s+2}$ bounded domain. Assume that $u^\nu$ is the weak
solution to the Navier-Stokes equations (1.1) with initial velocity
$u_0$. Let $u^0$ be the smooth solution to the problem (1.3) with
the same initial data as in Theorem 2.8. Then there exists a unique
boundary layer profile $u^b$ solving (3.7)-(3.8) with the following
regularities for $p>2$
\begin{equation}\label{3.9}\tag{3.9}
\begin{aligned}
& u^b\in L^\infty(0,T; W^{k,s,0,p});\\
&u^b\in C([0,T]; H^{k,s-2,1})\cap L^\infty(0,T;H^{k,s-2,2}\cap H^{k,s-1,1})\cap L^2(0,T;H^{k,s-2,3});\\
&\p_t u^b\in L^\infty(0,T; H^{k,s-2,0})\cap L^2(0,T; H^{k,s-2,1}\cap H^{k,s-1,0}).
\end{aligned}
\end{equation}
Consequently, there exists $\nu_0>0$, small enough, such that for
all $0<\nu\le\nu_0$, it holds that for all $p\in (3,6]$,
\begin{equation*}
\sup_{t\in [0,T]}\|u^\nu-u^0\|_{p}\le C\nu^{\frac{1}{2}+\frac{1}{2p}},\\
\sup_{t\in [0,T]}\|\rr\|_p\le C.
\end{equation*}
Here $C$ is independent of $\nu.$
\end{Theorem}
\begin{Remark}
In fact, it follows from Gagliado-Nirenberg inequality, Proposition 3.1 and  Theorem 2.8 that
\begin{equation*}
\sup_{t\in [0,T]}\|u^\nu-u^0\|_p\le
C\nu^{\frac{3}{10}+\frac{9}{10p}}
\end{equation*}
However, for $p>2$, the rate above is less than the one in Theorem 3.2.
\end{Remark}
Based on Theorem 3.2, the following main results can be proved.
\begin{Theorem}
Under the same assumptions in Theorem 3.2, there exists $\nu_0>0$,
suitably small, such that for all $0<\nu\le\nu_0$, it holds that
\begin{equation*}
\begin{aligned}
&\sup_{0\le t\le T}\|u^\nu-u^0\|_{H^1(\o)}\le C_1\nu^{\frac{1}{4}}, \\
&\sup_{0\le t\le T}\|u^\nu-u^0\|_\infty \le
C_2\nu^\frac{1}{2}
\end{aligned}
\end{equation*}
where $C_i$ is independent of $\nu$, $i=1,2.$
\end{Theorem}

\section{Estimates of boundary layers }
In this section, we will derive the main regularity estimates for the first order boundary layer profile $u^b$, including  $L^p-$estimates in $(x,z)$ and the estimates of time regularity.

The following lemmas will be applied in the rest of paper.

We start with some elementary estimates.

\begin{Lemma}
There exists a constant $C$ independent of $\nu$ such that for all $v\in L^p_z(\mathbb{R}_+;W_x^{2,p}(\o))=W^{0,2,0,p}(\Omega\times\mathbb{R}_+),~ p>1$ which vanishes for $x$ outside the neighborhood $\Lambda$ of $\p\o$,
\begin{alignat}{12}
\|v(x,\frac{\varphi(x)}{\sqrt{\nu}})\|_{p}\le
C\nu^{\frac{1}{2p}}\|v\|_{0,1,0,p},
\end{alignat}
where $\varphi(x)$ is the smooth function defined as in section 3.
 \end{Lemma}
\begin{proof}
 (4.1) follows from a similar argument for Lemma 3 in
\cite{IS}.
\end{proof}
\begin{Lemma}
Let $u_0\in W^{s+1,p}(\o)$ with $u_0\cdot\vec{n}=0$, then
$f(x,t)=\frac{u^0\cdot \vec{n}}{\varphi(x)}\in C([0,T];
W^{s,p}(\Omega))\cap C^1([0,T]; W^{s-1,p}(\o)).$
\end{Lemma}
\begin{proof}
This conclusion follows by modifying sightly the proof of Lemma 4 in \cite{IS}.
\end{proof}

The following proposition shows the $L^p$-estimates of the higher
order derivative in the $x$-variable for the boundary layer profile
$u^b$.
\begin{Proposition}
Let $2\le p<\infty$ and $k\geq 1$. If $u_0\in W^{m+3,p}(\Omega)$
with $\nabla\cdot u_0=0$ and $ u_0\cdot \vec{n}=0, \C  u_0\times
\vec{n}=0 $ on $\p\o$,  then
\begin{equation}
u^b\in L^\infty(0,T; W^{k,m,0,p}(\o\times\R_+)).
\end{equation}
\end{Proposition}
\begin{proof}
This can be verified by induction. Set $g(x,t)=\C\,u^0\times\vec{n}$. Then $$g\in C([0,T]; W^{m+2,p}(\o))\cap C^1([0,T]; W^{m+1,p}(\o)).$$

At first, we consider the case $m=0.$ Multiply  (\ref{3.6}) by
$(1+z^{2k})|u^b|^{p-2}u^b$ and integrate in $x$ and $z$ to obtain
\begin{equation}\label{4.3}\tag{4.3}
\begin{aligned}
&\frac{1}{p}\frac{\mathrm{d}}{\mathrm{d}t}\iint_{\Omega\times
\R_+}(1+z^{2k})|u^b|^p\mathrm{d}x\mathrm{d}z+\iint_{\Omega\times
\R_+}(1+z^{2k})u^b\cdot\nabla
u^0|u^b|^{p-2}u^b\mathrm{d}x\mathrm{d}z+\\
&\iint_{\Omega\times R_+}(1+z^{2k})u^0\cdot\nabla_x
u^b|u^b|^{p-2}u^b\mathrm{d}x\mathrm{d}z+\iint_{\Omega\times
\R_+}(z+z^{2k+1})f\cdot\partial_zu^b|u^b|^{p-2}u^b\mathrm{d}x\mathrm{d}z\\
&-\iint_{\Omega\times
\R_+}(1+z^{2k})\partial_z^2u^b|u^b|^{p-2}u^b\mathrm{d}x\mathrm{d}z=0.\nonumber
\end{aligned}\end{equation}
Since $\nabla\cdot u^0=0$ and $u^0\cdot \vec{u}=0$ on
$\partial\Omega$, the third term on the left hand side vanishes.
Integrating by parts with respect to $z$ to the last term and using
(3.8) and the decay property of $u^0$ due to Proposition 3.1 yield
\begin{equation}\label{4.4}\tag{4.4}
\begin{aligned}
&\frac{1}{p}\frac{d}{dt}\|u^b\|^p_{k,0,0,p}+\iint_{\Omega\times
\R_+}(1+z^{2k})|\partial_zu^b|^2|u^b|^{p-2}\mathrm{d}x\mathrm{d}z\\
\leq &-\iint_{\Omega\times
\R_+}2kz^{2k-1}\partial_zu^b|u^b|^{p-2}u^b\mathrm{d}x\mathrm{d}z-\iint_{\Omega\times
\R_+}(1+z^{2k})u^b\cdot\nabla u^0|u^b|^{p-2}u^b\mathrm{d}x\mathrm{d}z\\
&-\frac{1}{p}\iint_{\Omega\times
\R_+}(1+(2k+1)z^{2k})f|u^b|^p\mathrm{d}x\mathrm{d}z
+\int_{\o}g(x,t)|u^b(x,t,0)|^{p-2}|u^b(x,t,0)|\dif x\\
\equiv & \mathbb{I}_1+\mathbb{I}_2+\mathbb{I}_3+\mathbb{I}_4.
\end{aligned}
\end{equation}
 Young's inequality implies that
\begin{equation}\label{4.5}\tag{4.5}
|\mathbb{I}_1|\leq \e\iint_{\Omega\times
\R_+}(1+z^{2k})|\partial_zu^b|^2|u^b|^{p-2}\mathrm{d}x\mathrm{d}z+\iint_{\Omega\times
\R_+}(1+z^{2k})|u^b|^{p}\mathrm{d}x\mathrm{d}z.
\end{equation}
Due to the regularity of $u^0$ and $f$, one can get that
\begin{equation}\label{4.6}\tag{4.6}
|\mathbb{I}_2|+|\mathbb{I}_3|\leq C\iint_{\Omega\times
\R_+}(1+z^{2k})|u^b|^{p}\mathrm{d}x\mathrm{d}z.
\end{equation}
The term $\mathbb{I}_4$ can be estimated as follows:
\begin{equation}\label{4.7}\tag{4.7}
\begin{aligned}
&|\mathbb{I}_4|\le \|g\|_p\left(\int_\o |u^b(x,t,0)|^p\dif x\right)^{\frac{p-1}{p}}\le C\|g\|_p\left(\iint_{\o\times\mathbb{R}_+ }|\p_z u^b||u^b|^{p-2}u^b\dif x\dif z\right)^{\frac{p-1}{p}}\\
&\le C\|g\|_p\left(\iint_{\o\times\mathbb{R}_+ }(1+z^{2k})|\p_z u^b|^2|u^b|^{p-2}\dif x\dif z\right)^{\frac{p-1}{2p}}\left(\iint_{\o\times\mathbb{R}_+ }(1+z^{2k})|u^b|^p\dif x\dif z\right)^{\frac{p-1}{2p}}\\
&\le \e\iint_{\Omega\times
R_+}(1+z^{2k})|\partial_zu^b|^2|u^b|^{p-2}\mathrm{d}x\mathrm{d}z+C\|g\|_p^p+C \iint_{\Omega\times
R_+}(1+z^{2k})|u^b|^{p}\mathrm{d}x\mathrm{d}z.
\end{aligned}
\end{equation}
Then, for $\e=\frac{1}{2}$, one gets from (4.4)-(4.7) that
\begin{equation*}
\frac{1}{p} \frac{\mathrm{d}}{\mathrm{d}t}\|u^b\|^p_{k,0,0,p}\leq C
\|u^b\|^p_{k,0,0,p}+ C.
\end{equation*}
Thus Gronwall's Lemma yields $$\sup_{t\in(0,T)}\|u^b\|_{k,0,0,p}\le
C,$$ which proves (4.2) for $m=0$.

Assume that when $s\leq m-1$, it holds that $u^b\in L^\infty(0,T; W^{k,s,0,p}(\Omega)),$ when
$u_0\in W^{s+3,p}$. Next we verify that for $s=m$, $u^b\in L^\infty(0,T; W^{k,s,0,p}(\Omega))$ holds
true.

To this end, one can apply the operator $\p^\alpha_x$ to (3.6) with $|\alpha|=m$, multiply the resulting identity by
$(1+z^{2k})|\partial^\alpha_xu^b|^{p-2}\partial^\alpha_xu^b$, and integrate in $x$ and $z$ to obtain
\begin{equation}\label{4.8}\tag{4.8}
\begin{aligned}
&\frac{1}{p}\frac{d}{dt}\|u^b\|^p_{k,m,0,p}=\iint_{\Omega\times
\R_+}(1+z^{2k})\partial^2_z\partial_x^\alpha u^b|\partial_x^\alpha
u^b|^{p-2}\partial_x^\alpha u^b\mathrm{d}x\mathrm{d}z \\
&-\iint_{\Omega\times
\R_+}(1+z^{2k})\partial_x^\alpha [(u^b\cdot\nabla u^0+u^0\cdot\nabla_x u^b)\times \vec{n}]|\partial_x^\alpha u^b|^{p-2}\partial_x^\alpha u^b\mathrm{d}x\mathrm{d}z\\
&-\iint_{\Omega\times\R_+}(1+z^{2k})\partial_x^\alpha
(fz\partial_zu^b)|\partial_x^\alpha u^b|^{p-2}\partial_x^\alpha
u^b\mathrm{d}x\mathrm{d}z=I_1+I_2+I_3.
\end{aligned}
\end{equation}
To estimate $I_1$, we integrate by parts with respect to $z$ to get
\begin{equation*}
\begin{aligned}
&I_1\leq-\iint_{\Omega\times \R_+}(1+z^{2k})(\partial_z\partial_x^\alpha
u^b)|\partial_x^\alpha u^b|^{p-2}\partial_z\partial_x^\alpha u^b\mathrm{d}x\mathrm{d}z\\
&-\iint_{\Omega\times \R_+}2kz^{2k-1}(\partial_z\partial_x^\alpha
u^b)|\partial_x^\alpha
u^b|^{p-2}\partial_x^\alpha u^b\mathrm{d}x\mathrm{d}z\\
&+\int_\o\p_z(\partial^\alpha_x u^b)|\partial_x^\alpha
u^b|^{p-2}\partial_x^\alpha u^b|_{z=0}\dif x\\
&=-\iint_{\Omega\times \R_+}(1+z^{2k})|\partial_z\partial_x^\alpha
u^b|^2|\partial_x^\alpha
u^b|^{p-2}\mathrm{d}x\mathrm{d}z-I_{1_1}+I_{1_2}.\nonumber
\end{aligned}\end{equation*}
It follows from Young's inequality that\\
\begin{equation*}|I_{1_1}|\leq\e\iint_{\Omega\times
\R_+}(1+z^{2k})|\partial_z\partial_x^\alpha u^b|^2|\partial_x^\alpha
u^b|^{p-2}\mathrm{d}x\mathrm{d}z+C\iint_{\Omega\times
R_+}(1+z^{2k})|\partial_x^\alpha u^b|^{p}\mathrm{d}x\mathrm{d}z
\end{equation*}
Since $\p_z(\partial^\alpha_x u^b)|_{z=0}=\partial^\alpha_x g(x,t)$, by the same argument in the estimates of $\mathbb{I}_4$, one can get
\begin{equation*}
\begin{aligned}
&|I_{1_2}|\le \e\iint_{\Omega\times
\R_+}(1+z^{2k})|\partial_z\partial_x^\alpha u^b|^2|\partial_x^\alpha
u^b|^{p-2}\mathrm{d}x\mathrm{d}z\\
&+C\iint_{\Omega\times
R_+}(1+z^{2k})|\partial_x^\alpha u^b|^{p}\mathrm{d}x\mathrm{d}z+C\|\partial^\alpha_x g(x,t)\|^p_{p}.
\end{aligned}
\end{equation*}
Thus,
\begin{equation}\label{4.9}\tag{4.9}
I_1\leq-(1-2\e)\iint_{\Omega\times
\R_+}(1+z^{2k})|\partial_z\partial_x^\alpha u^b|^2|\partial_x^\alpha
u^b|^{p-2}\mathrm{d}x\mathrm{d}z+C\|u^b\|^p_{k,m,0,p}+C.
\end{equation}
Now we turn to $I_2$. It follows from direct calculations (see [12]) that
\begin{equation}\label{4.10}\tag{4.10}
\p_x^\alpha(u\times \vec{n})=(\partial_x^\alpha u)\times \vec{n}+D_x^{|\alpha|-1}(u),
\end{equation}
where $D^\alpha_x(u)$ denotes a linear combination of components of
$u$ and their derivatives with respect to $x$ of order $\leq |\alpha|$ with coefficients consisting of components of $n$ and its derivatives. Then $I_2$ can be rewritten as
\begin{equation}\label{4.11}\tag{4.11}
\begin{aligned}
&I_2=\iint_{\Omega\times \R_+}(1+z^{2k})\{[\partial_x^\alpha (u^b\cdot \nabla u^0+u^0\cdot\nabla_x u^b)]\times \vec{n}\\[5mm]
&+D_x^{m-1}(u^b\cdot\nabla u^0+u^0\cdot\nabla_x u^b)\}
\cdot|\partial_x^\alpha u^b|^{p-2}\partial_x^\alpha u^b\mathrm{d}x\mathrm{d}z\\[3mm]
&=\iint_{\Omega\times \R_+}(1+z^{2k})\{[\partial_x^\alpha (u^b\cdot \nabla u^0+u^0\cdot\nabla_x u^b)]+D_x^{m-1}(u^b\cdot\nabla u^0+u^0\cdot\nabla_x u^b)\\[5mm]
&-[\partial_x^\alpha(u^b\cdot\nabla u^0+u^0\cdot\nabla_x u^b)\cdot
\vec{n}]\vec{n}\}|\partial_x^\alpha u^b|^{p-2}\partial_x^\alpha
u^b\mathrm{d}x\mathrm{d}z=J_1+J_2+J_3.\end{aligned}\end{equation}

First, we assume that $m\ge 3$ and recall the Leibniz' formula
$$D^\alpha (uv)=\sum_{\beta\le\alpha}\binom{\alpha}{\beta}D^\beta uD^{\alpha-\beta}v,$$
where $D^\alpha=\frac{\p^{|\alpha|}}{\p x_1^{\alpha_1}\cdots\p x_n^{\alpha_n}}$, $\alpha=(\alpha_1\cdots\alpha_n)$ and $\binom{\alpha}{\beta}=\frac{\alpha!}{\beta!(\alpha-\beta)!}$, $\beta\le\alpha$ means $\beta_i\le\alpha_i(i=1,\cdots,n)$.
Then
\begin{equation*}
\begin{aligned}
&J_1=\iint_{\Omega\times \R_+}(1+z^{2k})[\partial_x^\alpha (u^b\cdot \nabla u^0+u^0\cdot\nabla_x u^b)]|\partial_x^\alpha u^b|^{p-2}\partial_x^\alpha
u^b\mathrm{d}x\mathrm{d}z\\
&=\sum_{\beta\le\alpha}\binom{\alpha}{\beta}\iint_{\o\times\R_+}(1+z^{2k})\left(\p_x^\beta u^b\cdot\nabla\p_x^{\alpha-\beta} u^0+\p_x^\beta u^0\cdot\nabla\p_x^{\alpha-\beta}u^b\right)|\partial_x^\alpha u^b|^{p-2}\partial_x^\alpha
u^b\mathrm{d}x\mathrm{d}z\\
&=\sum_{\beta\le\alpha}\binom{\alpha}{\beta} (J^\beta_{11} + J^\beta_{12}).
\end{aligned}
\end{equation*}
For the terms of $|\beta|\ge1$, one has
\begin{equation*}
\begin{aligned}
|J^\beta_{11}|= & \left|\iint_{\o\times\R_+}(1+z^{2k})\p_x^\beta u^b\cdot\nabla\p_x^{\alpha-\beta} u^0|\partial_x^\alpha u^b|^{p-2}\partial_x^\alpha
u^b\mathrm{d}x\mathrm{d}z\right|\\
\le & C\|\nabla\p_x^{\alpha-\beta} u^0\|_{L^\infty}\|(1+z^{2k})^{\frac{1}{p}}u^b\|_{|\beta|,p}\|(1+z^{2k})^{\frac{1}{p}}\p_x^\alpha u^b\|^{p-1}_{p}\\
\le & C\|u^0\|_{m+3,p}^p\|(1+z^{2k})^{\frac{1}{p}} u^b\|^p_{|\beta|,p}+\|(1+z^{2k})^{\frac{1}{p}}\p_x^\alpha u^b\|^{p}_{p}.
\end{aligned}
\end{equation*}
For the terms of $\beta=0$, one gets from Sobolev's imbedding that
\begin{equation*}
\begin{aligned}
|J^0_{11}| = &
\left|\iint_{\o\times\R_+}(1+z^{2k})u^b\cdot\nabla\p_x^{\alpha}
u^0|\partial_x^\alpha u^b|^{p-2}\partial_x^\alpha
u^b\mathrm{d}x\mathrm{d}z\right|\\
\le & C\|\nabla\p_x^{\alpha} u^0\|_{2p}\|(1+z^{2k})^{\frac{1}{p}} u^b\|_{2p}\|(1+z^{2k})^{\frac{1}{p}}\p_x^\alpha u^b\|^{p-1}_{p}\\
\le & C\|u^0\|_{m+3,p}^p\|(1+z^{2k})^{\frac{1}{p}} u^b\|^{p}_{1,p}+\|(1+z^{2k})^{\frac{1}{p}}\p_x^\alpha u^b\|^{p}_{p},
\end{aligned}
\end{equation*}
since $2p\le\frac{3p}{3-p}.$

 Since $u^0\cdot \vec{n}=0$ on $\p\Omega$ and $\di u^0=0$, it follows that
\begin{equation*}
\begin{aligned}
\left|\iint_{\o\times\R_+}(1+z^{2k}) u^0\cdot\nabla\p_x^{\alpha}u^b|\partial_x^\alpha u^b|^{p-2}\partial_x^\alpha
u^b\mathrm{d}x\mathrm{d}z\right|=0.
\end{aligned}
\end{equation*}
Next, the other terms of $J^\beta_{12}$ for $\beta\neq 0$ can be estimated as
\begin{equation*}
\begin{aligned}
|J^\beta_{12}| = & \left|\iint_{\o\times\R_+}(1+z^{2k}) \p_x^{\beta}u^0\cdot\nabla\p_x^{\alpha-\beta}u^b|\partial_x^\alpha u^b|^{p-2}\partial_x^\alpha
u^b\mathrm{d}x\mathrm{d}z\right|\\
\le & C\|u^0\|^p_{m+3,p}\|(1+z^{2k})^{\frac{1}{p}} u^b\|^p_{m-|\beta|+1,p}+\|(1+z^{2k})^{\frac{1}{p}}\p_x^\alpha u^b\|^{p}_{p}.
\end{aligned}
\end{equation*}
Therefore, we can conclude that
$$|J_1|\le C \left( \sum_{\beta\le\alpha,\beta>0}\binom{\alpha}{\beta} (||(1+z^{2k})^{\frac{1}{p}} u^b||^p_{|\beta|,p} + ||(1+z^{2k})^{\frac{1}{p}} u^b||^p_{m-|\beta|+1,p}) \right) + C\|(1+z^{2k})^{\frac{1}{p}}\p_x^\alpha u^b\|^{p}_{p}.$$
Note that
\begin{equation*}
\begin{aligned}
|J_3|= &\left|\iint_{\Omega\times \R_+}(1+z^{2k})[\partial_x^\alpha (u^b\cdot \nabla u^0+u^0\cdot\nabla_x u^b)\cdot \vec{n}]\vec{n}|\partial_x^\alpha u^b|^{p-2}\partial_x^\alpha
u^b\mathrm{d}x\mathrm{d}z\right|\\
= &\left|\iint_{\o\times\R_+}(1+z^{2k})\left[\left(\sum_{\beta\le\alpha}\binom{\alpha}{\beta}\left(\p_x^\beta u^b\cdot\nabla\p_x^{\alpha-\beta} u^0\right.\right.\right.\right.\\
&\left.+\p_x^\beta u^0\cdot \nabla\p_x^{\alpha-\beta}u^b\right)\left.\right)\cdot \vec{n}\left.\right]\vec{n}|\partial_x^\alpha u^b|^{p-2}\partial_x^\alpha
u^b\mathrm{d}x\mathrm{d}z\left.\right|\\
\le & J_{31}+\left| \iint_{\o\times\R_+}(1+z^{2k})\left[\left( u^0\cdot\nabla\p_x^{\alpha}u^b\right)\cdot \vec{n}\right]\vec{n}|\partial_x^\alpha u^b|^{p-2}\partial_x^\alpha
u^b\mathrm{d}x\mathrm{d}z \right|\\
\equiv & J_{31}+J_{32}.
\end{aligned}
\end{equation*}
$J_{31}$ can be estimated exactly as for $J_1$. To estimate $J_{32}$, noting that $u^b\cdot\vec{n}=0$, so $\p^\alpha_x u^b\cdot\vec{n}=-D^{m-1}_x(u^b)$, one can estimate $J_{32}$ as follow
\begin{equation*}
\begin{aligned}
|J_{32}| = & | \iint_{\o\times\R_+}(1+z^{2k})
(u^0\cdot\nabla(\p_x^{\alpha}u^b\cdot \vec{n})|\partial_x^\alpha
u^b|^{p-2}(\partial_x^\alpha
u^b\cdot \vec{n})\\
&+u^0\cdot\nabla \vec{n}\cdot u^b D^{m-1}_xu^b |\partial_x^\alpha u^b|^{p-2})\mathrm{d}x\mathrm{d}z |\\
\leq &\iint_{\o\times\R_+}(1+z^{2k}) (|u^0\cdot\nabla D^{m-1}_xu^b|\partial_x^\alpha u^b|^{p-2}(\partial_x^\alpha u^b\cdot \vec{n})|\\
&+|u^0\cdot\nabla \vec{n}\cdot u^b D^{m-1}_xu^b |\partial_x^\alpha u^b|^{p-2}|)\mathrm{d}x\mathrm{d}z,\\
\end{aligned}
\end{equation*}
which can be handled as for $J_1$. Now, it is clear that $J_2$ can be treated as for $J_1$. Thus we have
\begin{equation}\label{4.12}\tag{4.12}
\begin{aligned}
I_2\le & C\left( \sum_{\beta\le\alpha,\beta>0}\binom{\alpha}{\beta} (||1+z^{2k})^{\frac{1}{p}} u^b||^p_{|\beta|,p} + ||(1+z^{2k})^{\frac{1}{p}} u^b||^p_{m-|\beta|+1} \right)\\
& +C\|(1+z^{2k})^{\frac{1}{p}} \p^\alpha_x u^b\|^p_p.
\end{aligned}\end{equation}
Next, one can estimate $I_3$ as
\begin{equation*}
\begin{aligned}
&|I_3|=\left|-\iint_{\Omega\times \R_+}(1+z^{2k})z\left(\sum_{\beta\le\alpha}\binom{\alpha}{\beta}
\p_x^\beta f\p_x^{\alpha-\beta}\partial_zu^b\right)|\partial_x^\alpha u^b|^{p-2}\partial_x^\alpha
u^b\mathrm{d}x\mathrm{d}z\right|\\
&\le C\left|\iint_{\Omega\times \R_+}(1+z^{2k})z f\partial_z|\partial_x^\alpha u^c|^{p}\mathrm{d}x\mathrm{d}z\right|\\
&+\left|-\iint_{\Omega\times \R_+}(1+z^{2k})z\left(\sum_{0<\beta\le\alpha}\binom{\alpha}{\beta}
\p_x^\beta f\partial_z\p_x^{\alpha-\beta}u^b\right)|\partial_x^\alpha u^b|^{p-2}\partial_x^\alpha
u^b\mathrm{d}x\mathrm{d}z\right|\\
&:=I_{31}+I_{32}.
\end{aligned}
\end{equation*}
By the regularity of $f$ and integrating by parts for $z$, one gets
\begin{equation*}\begin{aligned}
&|I_{31}|\leq C\left|\iint_{\Omega\times \R_+}z(1+z^{2k})\partial_z|\partial_x^\alpha u^b|^{p}\mathrm{d}x\mathrm{d}z \right|\\
&\leq C\iint_{\Omega\times \R_+}(1+(2k+1)z^{2k})|\partial_x^\alpha
u^b|^{p}\mathrm{d}x\mathrm{d}z\\
&\leq C\|u^b\|^p_{k,m,0,p}.
\end{aligned}\end{equation*}

Due to the regularity of $\partial^\beta_x f$, it follows from integration by parts in $z$ and Young's inequality that
\begin{equation*}\begin{aligned}
I_{32}=& | \sum_{0<\beta\le\alpha}\binom{\alpha}{\beta} \iint_{\Omega\times \R_+} [(1+(2k+1)z^{2k})\p^\beta_x f \p_x^{\alpha-\beta}u^b|\partial_x^\alpha u^b|^{p-2}\partial_x^\alpha u^b\\
&+z(1+z^{2k}) \p^\beta_x f \p^{\alpha-\beta}_x u^b \p_z (|\p^\alpha_x u^b|^{p-2} \cdot \p^\alpha_x u^b)]\mathrm{d}x\mathrm{d}z|\\
\le & C\sum_{0<\beta\le\alpha}\binom{\alpha}{\beta} (\iint_{\Omega\times \R_+}(1+(2k+1)z^{2k}) |\partial_x^{\alpha-\beta}
u^b| |\partial_x^\alpha u^b|^{p-1} \mathrm{d}x\mathrm{d}z\\
& +\iint_{\Omega\times\R_+}(z+z^{2k+1}) |\p^{\alpha-\beta}_x u^b| |\p^\alpha_x u^b|^{p-2} |\p^\alpha_x \p_z u^b| \mathrm{d}x\mathrm{d}z)\\
\le & C \sum_{0<\beta\le\alpha}\binom{\alpha}{\beta} (\|u^b\|^p_{k,|\alpha-\beta|,0,p}+\|u^b\|^p_{k+{\frac{p}{2}},m-|\beta|,0,p})+C\|u^b\|^p_{k,m,0,p}\\
&+
\e\iint_{\Omega\times
\R_+}(1+z^{2k})|\partial_x^\alpha \partial_z u^b|^2 |\partial_x^\alpha
u^b|^{p-2} \mathrm{d}x\mathrm{d}z.
\end{aligned}\end{equation*}
Thus, collecting all the estimates above and using (4.8), one has by choosing $\e=\frac{1}{3}$ to get
\begin{equation}\label{4.13}\tag{4.13}
\begin{aligned}
&\frac{1}{p}\frac{d}{dt}\|u^b\|^p_{k,m,0,p}\\
\le & C \sum_{0<\beta\le\alpha}\binom{\alpha}{\beta} (\|u^b\|^p_{k,m-|\beta|,0,p} + \|u^b\|^p_{k+{\frac{p}{2}},m-|\beta|,0,p})+ C \|u^b\|^p_{k,m,0,p}.
\end{aligned}
\end{equation}
which yields (4.2) for $s=m$ by the induction assumptions and Gronwall's Lemma. This completes the proof of Proposition 4.3.
\end{proof}

\begin{Remark}
In this proof, it is required that $u_0\in W^{s+3,p}(\o)$. However, for $p=2$, one can prove that if  $u_0\in H^{m+1}(\Omega)\cap L^2_\sigma(\o)$ then $u^b\in L^\infty(0,T; H^{k,m,0})$.
In fact, for $m\le 2$, one can refer to the proof in \cite{IS}. However, if $m>2$, the proof can be done by modifying the $H_x^2$-estimates in \cite{IS} or by modifying the following analysis in Lemmas 4.5 and 4.6.
\end{Remark}
It seems difficult to get $L^p$-estimates of derivatives of the boundary layer profile in $(t,z)-$variable, but we can obtain the following estimates which improve the regularity with respect to $(t,z)$-variable in \cite{IS}.
\begin{Lemma}
If $u_0\in H^{m+1}(\o)\cap L^2_\sigma(\o),~\, m\ge 2$, with $u_0\cdot n|_{\p\o}=0,$ then $u^b\in L^{\infty}(0,T;H^{k,m,1}(\o\times\R_+))$ and $\p_t u^b\in L^2(0,T;H^{k,m,0}(\Omega\times\R_+)).$
\end{Lemma}
\begin{proof}
Apply $\p_x^\alpha$ with $|\alpha|=m$ to (\ref{3.6}) to get
\begin{equation}\label{4.14}\tag{4.14}
\p_t\p_x^\alpha u^b-\p_z^2\p^\alpha_x
u^b+\p_x^\alpha(fz\p_zu^b)+\p_x^\alpha[(u^b\cdot\nabla
u^0+u^0\cdot\nabla u^b)\times \vec{n}]=0.
\end{equation}

Multiplying (\ref{4.14}) by $(1+z^{2k})\p_t\p_x^\alpha u^b$ and integrating over $\o\times\R_+$, one has
\begin{equation*}
\begin{aligned}
&0=\iint_{\o\times\R_+}(1+z^{2k})|\p_t\p_x^\alpha u^b|^2\dif x\dif z-\iint_{\o\times\R_+}(1+z^{2k})\p_t\p_x^\alpha u^b\p_z^2\p^\alpha_x u^b\dif x\dif z\\
&+\iint_{\o\times\R_+}(1+z^{2k})\p_t\p_x^\alpha u^b\p_x^\alpha(fz\p_zu^b)\dif x\dif z \\
&+\iint_{\o\times\R_+} (1+z^{2k})\p_t\p_x^\alpha u^b\p_x^\alpha[(u^b\cdot\nabla u^0+u^0\cdot\nabla u^b)\times \vec{n}]\\
&=\iint_{\o\times\R_+}(1+z^{2k})|\p_t\p_x^\alpha u^b|^2\dif x\dif
z+K_1+K_2+K_3.
\end{aligned}
\end{equation*}
We estimate $K_1$ first. In fact, integration by parts with respect
to $z$ yields
\begin{equation*}
\begin{aligned}
K_1= &\frac{1}{2}\frac{\dif}{\dif t}\iint_{\o\times\R_+}(1+z^{2k})|\p_x^\alpha \p_z u^b|^2\dif x\dif z\\
&+\iint_{\o\times\R_+}2kz^{2k-1}\p_z\p_x^\alpha u^b\cdot\p_t\p_x^\alpha u^b \dif x\dif z-\int_{\o}\p_z\p_x^\alpha u^b (x,t,0)\p_t\p_x^\alpha u^b(x,t,0)\dif x\\
\le &\frac{1}{2}\frac{\dif}{\dif t}\iint_{\o\times\R_+}(1+z^{2k})|\p_x^\alpha \p_z u^b|^2\dif x\dif z+C\|\p_z\p_x^\alpha u^b\|^2_{k,0,0}+\e\|\p_t\p_x^\alpha u^b\|^2_{k,0,0}\\
&-\int_{\o}\p_x^\alpha g\cdot\p_t\p_x^\alpha u^b(x,t,0)\dif x.
\end{aligned}
\end{equation*}
The last term on the right can be estimated as follows:
 \begin{equation*}
 \begin{aligned}
&-\int_{\o}\p_x^\alpha g\cdot\p_t\p_x^\alpha u^b(x,t,0)\dif x=-\frac{\dif}{\dif t}\int_{\o}\p_x^\alpha g\p_x^\alpha u^b|_{z=0}\dif x+\int_{\o}\p_t\p_x^\alpha g\p_x^\alpha u^b|_{z=0}\dif x\\
 &=\frac{\dif}{\dif t}\iint_{\o\times\R_+}\p_x^\alpha g\p_z\p_x^\alpha u^b\dif x\dif z-\iint_{\o\times\R_+}\p_t\p_x^\alpha g\p_z\p_x^\alpha u^b\dif x\dif z\\
 &\le \frac{\dif}{\dif t}\iint_{\o\times\R_+}\p_x^\alpha g\p_z\p_x^\alpha u^b\dif x\dif z+\|\p_z\p_x^\alpha u^b\|_{k,0,0}\|\p_t\p_x^\alpha g(1+z^{2k})^{-\frac{1}{2}}\|_{L^2(\o\times\R_+)}.
 \end{aligned}
 \end{equation*}
Therefore,
\begin{equation}\label{4.15}\tag{4.15}
\begin{aligned}
K_1\le & \frac{\dif}{\dif t}\iint_{\o\times\R_+}\left(\frac{1}{2}(1+z^{2k})|\p_x^\alpha \p_z u^b|^2+\p_x^\alpha g\p_z\p_x^\alpha u^b\right)\dif x\dif z\\
&+C\|\p_z\p_x^\alpha u^b\|^2_{k,0,0}+\e\|\p_t\p_x^\alpha u^b\|^2_{k,0,0}+\|\p_t\p_x^\alpha g(1+z^{2k})^{-\frac{1}{2}}\|^2_{L^2(\o\times\R_+)}.
\end{aligned}
\end{equation}
Next, we estimate $K_2$. Note that
\begin{equation*}
\begin{aligned}
K_2= &\iint_{\o\times\R_+}(1+z^{2k})\p_t\p_x^\alpha u^b\p_x^\alpha(fz\p_zu^b)\dif x\dif z\\
= & \iint_{\o\times\R_+}(1+z^{2k})\p_t\p_x^\alpha u^b z\sum_{\beta\le\alpha}\binom{\alpha}{\beta}\p_x^\beta f \p_x^{\alpha-\beta}\p_z u^b\dif x\dif z.
\end{aligned}
\end{equation*}
The terms on the right hand side above can be handled as follows.

For the case $\beta=0$, one has
\begin{equation*}
\begin{aligned}
&\left| \iint_{\o\times\R_+}(1+z^{2k})\p_t\p_x^\alpha u^b z\sum_{\beta\le\alpha}f \p_x^{\alpha}\p_z u^b\dif x\dif z\right|\\
\le & \|f\|_{\infty}\|(1+z^{2k})^{\frac{1}{2}}\p_t\p_x^\alpha u^b
\|_{2}\|(1+z^{2(k+1)})^{\frac{1}{2}}\p_z\p_x^\alpha u^b \|_{2}
\end{aligned}
\end{equation*}

For $|\beta|=1$, one can get
\begin{equation*}
\begin{aligned}
&\left|\iint_{\o\times\R_+}(1+z^{2k})\p_t\p_x^\alpha u^b z\sum_{\beta\le\alpha}\p_x^\beta f \p_x^{\alpha-\beta}\p_z u^b\dif x\dif z\right|\\
\le & C\int_{\R_+}(1+z^{2k})z\|\p_t\p_x^\alpha u^b\|_{2} \|\p_x^\beta f \|_{6}\|\p_x^{\alpha-\beta}\p_z u^b\|_{3}\dif z\\
\le & C\sum_{\beta\leq\alpha\atop{|\beta|\geq p}}\|(1+z^{2k})^{\frac{1}{2}}\p_t\p_x^\alpha u^b \|_{2}\|f\|_{2,2}\|(1+z^{2(k+2)})^{\frac{1}{2}}\p_z\p_x^{\alpha-\beta} u^b \|^{\frac{1}{2}}_{2} \|(1+z^{2k})^{\frac{1}{2}}\p_z\p_x^{\alpha-\beta} u^b \|^{\frac{1}{2}}_{1,2}
\end{aligned}
\end{equation*}
here one has used the interpolation inequality $\|u\|_{3}\le \|u\|^\frac{1}{2}_{2}\|u\|^{\frac{1}{2}}_{1,2}$.

In the case that $|\beta|\ge 2$, one has
\begin{equation*}
\begin{aligned}
&\left| \iint_{\o\times\R_+}(1+z^{2k})\p_t\p_x^\alpha u^b z\sum_{\beta\le\alpha}\p_x^\beta f \p_x^{\alpha-\beta}\p_z u^b\dif x\dif z\right|\\
\le & C\int_{\R_+}(1+z^{2k})z\|\p_t\p_x^\alpha u^b\|_{2} \|\p_x^\beta f \|_{2}\|\p_x^{\alpha-\beta}\p_z u^b\|_{\infty}\dif z\\
\le & C\|(1+z^{2k})^{\frac{1}{2}}\p_t\p_x^\alpha u^b \|_{2}\|f\|_{\beta,2}\|(1+z^{2(k+2)})^{\frac{1}{2}}\p_z\p_x^{\alpha-\beta} u^b \|^{\frac{1}{2}}_{1,2} \|(1+z^{2k})^{\frac{1}{2}}\p_z\p_x^{\alpha-\beta} u^b \|^{\frac{1}{2}}_{2,2}
\end{aligned}
\end{equation*}
where one has used the interpolation inequality $\|u\|_{\infty}\le\|u\|^{\frac{1}{2}}_{1,2}\|u\|^\frac{1}{2}_{2,2}.$

Therefore,
\begin{equation}\label{4.16}\tag{4.16}
\begin{aligned}
K_2\le C(\|u^b\|^2_{k+1,m,1}+\|u^b\|^2_{k+2,m-1,1}+\|u^b\|^2_{k,m,1})+\e\|\p_t\p_x^\alpha u^b\|^2_{k,0,0}.
\end{aligned}
\end{equation}

Using the similar argument as above and the proof of $I_2$ in
Proposition 4.3, we can obtain that
\begin{equation}\label{4.17}\tag{4.17}
K_3\le C(\|u^b\|^2_{k,m-1,1}+\|u^b\|^2_{k,m,1})+\e\|\p_t\p_x^\alpha u^b\|^2_{k,0,0}.
\end{equation}

Due to Remark 4.4, it holds that
$\|u^b\|^2_{L^2(0,T;H^{k,m,1}(\o\times\R_+)}\le C.$ It follows from
the Gronwall's lemma, H\"{o}lder inequality and the choice of
$\e=\frac{1}{6}$ that
$$\|\p_t u^b\|_{L^2(0,T;H^{k,m,0}(\o\times\R_+))}\le C,~\,\| u^b\|_{L^\infty(0,T;H^{k,m,1}(\o\times\R_+))}\le C. $$
The proof is completed.
\end{proof}
In fact, the regularity of $\p_t u^b$ can be improved further, if the initial data is regular enough.
\begin{Lemma}
If $u_0\in H^{m+2}(\o)\cap L_\sigma^2(\o), m\ge 2$ with $u_0\cdot
\vec{n}|_{\p\o}=0,$ then it holds that
\begin{equation}\label{4.18}\tag{4.18}
\p_t u^b\in L^{\infty}(0,T; H^{k,m,0}(\o\times\R_+))\cap L^2(0,T;H^{k,m,1} (\o\times\R_+)).
\end{equation}
Consequently,
\begin{equation}\label{4.19}\tag{4.19}
u^b\in L^{\infty}(0,T; H^{k,m,2}(\o\times\R_+))\cap L^2(0,T;H^{k,m,3} (\o\times\R_+)).
\end{equation}
\end{Lemma}
\begin{proof}
It follows from (3.6) that
\begin{equation}\label{4.20}\tag{4.20}
\p_t(\p_t\p_x^\alpha u^b)-\p_z^2(\p_t\p_x^\alpha u^b)+\p_t\p_x^\alpha(fz\p_z u^b)+\p_t\p_x^\alpha[(u^0\cdot\nabla u^b+u^b\cdot\nabla u^0)\times \vec{n}]=0.
\end{equation}
Multiply (4.20) by $(1+z^{2k})\p_t\p_x^\alpha u^b$ and integrate over $\o\times\R_+$ to get
\begin{equation}\label{4.21}\tag{4.21}
\begin{aligned}
&\frac{1}{2} \frac{\dif}{\dif t}\iint_{\o\times\R_+}(1+z^{2k})|\p_t \p_x^\alpha u^b|^2\dif x\dif z-\iint_{\o\times\R_+}(1+z^{2k})\p_z^2(\p_t\p_x^\alpha u^b)\p_t\p_x^\alpha u^b\dif x\dif z\\
&+\iint_{\o\times\R_+}(1+z^{2k})\p_t\p_x^\alpha(fz\p_z u^b)\p_t\p_x^\alpha u^b\dif x\dif z\\
&+\iint_{\o\times\R_+}(1+z^{2k})\p_t\p_x^\alpha[(u^0\cdot\nabla u^b+u^b\cdot\nabla u^0)\times \vec{n}]\p_t\p_x^\alpha u^b\dif x\dif z=0.
\end{aligned}
\end{equation}
We treat the case $|\alpha|=0$ first. Note that
\begin{equation*}
\begin{aligned}
&-\iint_{\o\times\R_+}(1+z^{2k})\p^2_z (\p_t u^b) \p_t u^b\dif x\dif z = \iint_{\o\times\R_+}(1+z^{2k}) |\p_z \p_t u^b|^2 \dif x\dif z\\
& + 2k \iint_{\o\times\R_+} z^{2k-1} \p_z \p_t u^b \p_t u^b \dif x\dif z - \int_\Omega \p_t \p_z u^b \p_t u^b|_{z=0} \dif x\\
\geq &\iint_{\o\times\R_+}(1+z^{2k}) |\p_z \p_t u^b|^2 \dif x\dif z - C||(1+z^{2k})^{\frac{1}{2}} \p_z \p_t u^b||_2 ||(1+z^{2k})^{\frac{1}{2}} \p_t u^b||_2\\
& -\iint_{\o\times\R_+} |\p_t g \p_t \p_z u^b| \dif x\dif z\\
\geq & \iint_{\o\times\R_+} (1+z^{2k}) |\p_z \p_t u^b|^2 \dif x\dif z -C ||(1+z^{2k})^{\frac{1}{2}} \p_z\p_t u^b||_2 \cdot ||(1+z^{2k})^{\frac{1}{2}} \p_t u^b||_2\\
&-C||(1+z^{2k})^{\frac{1}{2}} \p_z\p_t u^b||_2 ||(1+z^{2k})^{-\frac{1}{2}} \p_t g||_2,
\end{aligned}
\end{equation*}
and
\begin{equation*}
\begin{aligned}
&\left|\iint_{\o\times\R_+}(1+z^{2k})\p_t(fz\p_z u^b)\p_t u^b\dif x\dif z\right|\\
\le & \left|\iint_{\o\times\R_+}(1+z^{2k})\p_t fz\p_z u^b\p_t u^b\dif x\dif z\right|\\
&+ \left|\iint_{\o\times\R_+}(1+z^{2k})fz\p_z \p_t u^b\p_t u^b\dif x\dif z\right|\\
\le & C\|\p_t f\|_{L^\infty}\|(1+z^{2k+2})^{\frac{1}{2}} \p_t u^b\|_{2}\|(1+z^{2k})^{\frac{1}{2}} \p_z u^b\|_{2}\\
&+C\| f\|_{L^\infty}\|(1+z^{2k+2})^{\frac{1}{2}} \p_tu^b\|_{2}\|\p_t  \p_z u^b(1+z^{2k})^\frac{1}{2}\|_{2}\\
\le & C\|\p_tu^b\|^2_{k+1,0,0}+\|u^b\|^2_{k,0,1}+\e\|\p_t  \p_z u^b\|^2_{k,0,0}.
\end{aligned}
\end{equation*}
Similarly, one can get
 \begin{equation*}
\begin{aligned}
&\left|\iint_{\o\times\R_+}(1+z^{2k})\p_t[(u^0\cdot\nabla u^b+u^b\cdot\nabla u^0)\times \vec{n}]\p_t u^b\dif x\dif z\right|\\
=& \left| \iint_{\o\times\R_+}(1+z^{2k})\left[(\p_tu^0\cdot\nabla u^b+u^0\cdot\nabla\p_t u^b+\p_t u^b\cdot\nabla u^0+u^b\nabla\p_t u^0)\times \vec{n}\right]\p_tu^b\dif x\dif z\right|\\
\le & C(\|\p_tu^c\|^2_{k,2,0}+\| u^c\|^2_{k,1,0}).
\end{aligned}
\end{equation*}
It follows from these, the Gronwall's Lemma and Lemma 4.5 that
\begin{equation}\label{4.22}\tag{4.22}
\begin{aligned}
\sup_{0\le t\le T}\|\p_t u^b(t)\|^2_{k,0,0}+\int_0^T\iint_{\o\times\R_+}(1+z^{2k})|\p_t\p_z u^b|^2\dif x\dif z\dif t\le C+\lim_{t\to 0}\|\p_t u^b(t)\|^2_{k,0,0}.
\end{aligned}
\end{equation}
By induction, as in argument for Proposition 4.3, we can obtain
\begin{equation}\label{4.23}\tag{4.23}
\begin{aligned}
\sup_{0\le t\le T}\|\p_t u^b(t)\|^2_{k,m,0}+\|\p_t u^b|^2_{L^2(0,T;H^{k,m,1}(\o\times\R_+))}\le C+\lim_{t\to 0}\|\p_t u^b(t)\|^2_{k,m,0}.
\end{aligned}
\end{equation}
Since $u^b(x,0,z)=0$, thus $\p_x^\alpha\p^j_z u^b(x,0,z)=0,$  for all $\alpha,j$. Therefore, taking limit in both sides of (4.22)-(4.23) as $t\to 0$,  we can get $\p_t\p_x^\alpha u^b(x,t,z)\to 0$ in a.e. $\o\times\R_+$ as $t\to 0.$  Therefore, the last terms on the right side of (4.22) and (4.23) vanish. This proves (4.18).

To prove (4.19), one uses (3.16) again, (4.18) and Lemma 4.5 to get
\begin{equation}\label{4.24}\tag{4.24}
u^b\in L^{\infty}(0,T; H^{k,m,2}(\o\times\R_+)).
\end{equation}
Similarly, differentiating (3.6) in $z$, we can then use (4.18),
(4.24) and Lemma 4.5 to obtain the second part of (4.19). Thus Lemma
4.6 is proved.
\end{proof}
\section{$L^p-$uniform bound of the remainder $R^\nu$ for $3<p\le6$}
In this section, we give the $L^p$-estimates of the remainder $R^\nu$ in (3.1). Note that the
remainder $R^\nu$ satisfies the following equation ([12]):
\begin{equation}\label{5.1}
\begin{aligned}
&\partial_tR^\nu-\nu \triangle R^\nu+u^\nu\cdot\nabla
R^\nu+R^\nu\cdot\nabla u^0+\sqrt{\nu}R^\nu\cdot \vec{n}\partial_z
v+R^\nu\cdot \vec{n}\partial_zu^b+\sqrt{\nu}R^\nu\cdot
\nabla_xu^b\\
=&-\partial_t v+\triangle
u^0+\sqrt{\nu}\triangle_xu^b+2\vec{n}\cdot\nabla_x\partial_z
u^b+\nu\triangle_x[v(x,\frac{\varphi(x)}{\sqrt{\nu}})]-u^\nu\cdot\nabla_x
v-v\cdot\nabla u^0\\
&-\frac{1}{\sqrt{\nu}}u^0\cdot \vec{n}\partial_zv
-\sqrt{\nu}v\cdot \vec{n}\partial_zv-v\cdot
\vec{n}\partial_zu^b-u^b\cdot\nabla_xu^b+\triangle\varphi\cdot\partial_zu^b-\sqrt{\nu}v\cdot\nabla_x
u^b\\
&+\nabla_xq+\nabla_x\kappa:=R.H.S. ~~~~~~~~~~~~~~~~~~~~~\mbox{in}~~\Omega,&\\
&\mathrm{div}R^\nu=-
 \mathrm{div}_x v(t,x,\frac{\varphi(x)}{\sqrt{\nu}}) ~~~~~~~~~~~~~~~~~~\mbox{in}~~\Omega,&
\end{aligned}
\end{equation}
with the boundary conditions:
\begin{alignat}{12}
& R^\nu\cdot \vec{n}(x)+v(t,x,0)\cdot \vec{n}(x)=0,~~~~~\mbox{for}~~x\in\partial\Omega,\label{5.2}\\
&\C\,\rr\times\vec{n}+\frac{1}{\sqrt{\nu}}\C_x u^b
(t,x,0)\times\vec{n}+\C_x v(t,x,0)\times \vec{n}=0
~\mbox{for}~~x\in\partial\Omega.\label{5.3}
\end{alignat}
The initial data for $R^\nu$ is
\begin{equation}\label{5.4}~~~~~~~~~~~~~ R^\nu(0,x)=0,
~~~~~~\mbox{for}~~x\in\Omega~~~~~~~~~.
\end{equation}
Set $$b(t,x)=\frac{1}{\sqrt{\nu}}u^b(t,x,0)+ v(t,x,0)
.$$

In the sequel, we need the following anisotropic Sobolev embedding result whose proof is given in [12].
\begin{Lemma}
Let $U(x,z)$ be a sufficiently regular function defined on $\Omega\times\R_+$. Assume that either $ 2\le p<\infty$, $m\ge\frac{3}{2}-\frac{3}{p}$ or $p=\infty$, $m>\frac{3}{2}$. Then
\begin{equation}\label{5.5}
\|U(x,\frac{\varphi(x)}{\sqrt{\nu}})\|_p\le C\|U\|_{1,m,1}.
\end{equation}
\end{Lemma}

We now begin to derive the $L^p$-estimate of $R^\nu$. To this end ,
we need the following Weyl decomposition of the space $L^p(\o)$:
\begin{Lemma}
Let $p\geq 2$. Set $G_p(\o)=\{u\in L^p(\o):u=\nabla q, q\in W^{1,p}(\o)\}$ and $J_p(\o)=\{u\in L^p(\o):\di u=0$ in $\o$, $u\cdot \vec{n}=0$ on $\p\o\}$. Then
$$L^p(\o)=G_p(\o)\bigoplus J_p(\o)$$
and the projections of an arbitrary vector field $u(x)$ to the above subspaces are defined respectively by the formulas
\begin{equation}\label{5.6}
\begin{aligned}
\P_G u&=-\nabla\int_\o\nabla_y N(x,y)\cdot u(y) dy\\
\P_J u&= u+\nabla\int_\o\nabla_y N(x,y)\cdot u(y) dy.
\end{aligned}
\end{equation}
with the following estimates:
\begin{equation*}
\begin{aligned}
\|\P_G u\|_{l,p}+\|\P_J u\|_{l,p}\le C\|u\|_{l,p}
\end{aligned}
\end{equation*}
where $l<r$ if $\p\o\in C^{r+1}$ and $u\in W^{l,p}(\o)$.
\end{Lemma}
\begin{proof}
See\cite{SO1}
\end{proof}
For simplicity, we set $\mathbb{P}=\mathbb{P}_J$ and decompose $R^\nu=\mathbb{P}R^\nu+(I-\mathbb{P})R^\nu$. It can be shown easily that $(I-\mathbb{P})R^\nu$ is bounded in $W^{1,p}(\Omega)$
independent of $\nu$.

\begin{Lemma} $(I-\mathbb{P})R^\nu$ is uniformly
bounded in $L^\infty(0,T;W^{1,p}(\Omega))$ for $3<p\le 6$, that is,
\begin{equation*}
\|(I-\mathbb{P})R^\nu\|_{1,p}\le C\|u^b\|_{1,3,0}.
\end{equation*}
\end{Lemma}
\begin{proof}
This conclusion follows easily by using the standard $L^p$ estimates for elliptic equations, or refer to \cite{IS}.
\end{proof}

The main part of this section is to bound $\|\mathbb{P}R^\nu\|_p$ independent of $\nu$. Indeed, we have
\begin{Lemma}
Let $3<p\leq 6$. There exists a positive constant $\nu_0\in(0,1)$ such that for all $\nu\in(0,\nu_0]$, it holds that
\begin{equation}\tag{5.8}
\sup_{0\leq t\leq T} \|\mathbb{P} R^\nu\|^p_p + C_0\nu\int^T_0 \|\nabla |\mathbb{P} R^\nu|^{\frac{p}{2}} \|^2_2 dt \leq C
\end{equation}
with positive constants $C_0$ and $C$ independent of $\nu$.
\end{Lemma}

The rest of this section is devoted to the proof of Lemma 5.4. In order to avoid estimating the unknown pressure term $\nabla_x k$, one needs to take inner product of (5.1) with $\P(|\P\rr|^{p-2}\P\rr)$. To simplify the computation, due to Lemma 5.2, we rewrite $\P(|\P\rr|^{p-2}\P\rr)$ as
\begin{equation}\tag{5.9}
\P(|\P\rr|^{p-2}\P\rr)=|\P\rr|^{p-2}\P\rr+\nabla Q
\end{equation}
with $Q=\int_\o\nabla_y N(x,y)\cdot |\P\rr|^{p-2}\P\rr dy$ satisfying
\begin{equation}\tag{5.10}
\left\{ \begin{aligned}
& \|\nabla Q\|_s \le C\|\P\rr\|^{(p-1)}_{(p-1)s}, \qquad 1<s<\infty,\\
& \|\nabla^2 Q\|_{\frac{p}{p-1}} \le C\|\nabla\left(|\P\rr|^{\frac{p}{2}}\right)\|_2 \cdot \|\P\rr\|_p^{\frac{p-2}{2}},\\
&\frac{\p Q}{\p \vec{n}}= 0~\mbox{on}~\p\o.
\end{aligned}
\right.
\end{equation}

It follows from (5.1), (5.9)-(5.10) that
\begin{equation}\tag{5.11}
\frac{1}{p}\frac{\dif}{\dif t}\|\mathbb{P}R^\nu\|_p^p-\nu\int_\Omega\triangle
R^\nu\cdot(|\mathbb{P}R^\nu|^{p-2}\mathbb{P}R^\nu+\nabla Q)
\mathrm{d}x = \sum\limits_{k=1}^{19} B_k,
\end{equation}
where
\begin{equation*}
\begin{array}{ll}
\displaystyle B_1=-\int_\Omega u^\nu\cdot\nabla
R^\nu\cdot\P(|\mathbb{P}R^\nu|^{p-2}\mathbb{P}R^\nu)
\mathrm{d}x, & \displaystyle B_2=-\int_\Omega R^\nu\cdot\nabla
u^0\P(|\mathbb{P}R^\nu|^{p-2}\mathbb{P}R^\nu)
\mathrm{d}x,
\\[3mm]
\displaystyle B_3=-\sqrt{\nu}\int_\Omega R^\nu\cdot \vec{n}\partial_zv
\cdot\P(|\mathbb{P}R^\nu|^{p-2}\mathbb{P}R^\nu)
\mathrm{d}x,
& \displaystyle B_4=-\int_\Omega R^\nu\cdot
\vec{n}\partial_zu^b\cdot\P(|\mathbb{P}R^\nu|^{p-2}\mathbb{P}R^\nu)
\mathrm{d}x,
\\[3mm]
\displaystyle B_5=-\sqrt{\nu}\int_\Omega R^\nu\cdot\nabla_xu^b
\cdot\P(|\mathbb{P}R^\nu|^{p-2}\mathbb{P}R^\nu)
\mathrm{d}x,
& \displaystyle B_6=-\int_\Omega
\partial_tv\cdot\P(|\mathbb{P}R^\nu|^{p-2}\mathbb{P}R^\nu)
\mathrm{d}x,\\[3mm]
\displaystyle B_7=\int_\Omega \triangle u^0
\cdot\P(|\mathbb{P}R^\nu|^{p-2}\mathbb{P}R^\nu)
\mathrm{d}x,
 & \displaystyle  B_8=\sqrt{\nu}\int_\Omega
\triangle_xu^b\cdot\P(|\mathbb{P}R^\nu|^{p-2}\mathbb{P}R^\nu)
\mathrm{d}x,\\[3mm]
\displaystyle B_9=\int_\Omega
2\vec{n}\cdot\nabla_x\partial_zu^b\cdot\P(|\mathbb{P}R^\nu|^{p-2}\mathbb{P}R^\nu)
\mathrm{d}x,
& B_{10}=\nu\int_\Omega
\triangle_x[v(x,\frac{\varphi(x)}{\sqrt{\nu}})]\cdot\P(|\mathbb{P}R^\nu|^{p-2}\mathbb{P}R^\nu)
\mathrm{d}x,
\end{array}
\end{equation*}
\begin{equation*}
\begin{array}{ll}
\displaystyle B_{11}=-\int_\Omega
u^\nu\cdot\nabla_xv\cdot\P(|\mathbb{P}R^\nu|^{p-2}\mathbb{P}R^\nu)
\mathrm{d}x,
 & B_{12}=-\int_\Omega v\cdot\nabla
u^0\cdot\P(|\mathbb{P}R^\nu|^{p-2}\mathbb{P}R^\nu)
\mathrm{d}x, \\[3mm]
\displaystyle B_{13}=-\frac{1}{\sqrt{\nu}}\int_\Omega u^0\cdot
\vec{n}\partial_zv\cdot\P(|\mathbb{P}R^\nu|^{p-2}\mathbb{P}R^\nu)
\mathrm{d}x,
 & B_{14}=-\sqrt{\nu}\int_\Omega v\cdot
\vec{n}\partial_z v\cdot\P(|\mathbb{P}R^\nu|^{p-2}\mathbb{P}R^\nu)
\mathrm{d}x,\\[3mm]
\displaystyle B_{15}=-\int_\Omega v\cdot
\vec{n}\partial_zu^b\cdot\P(|\mathbb{P}R^\nu|^{p-2}\mathbb{P}R^\nu)
\mathrm{d}x,
& B_{16}=-\int_\Omega u^b\cdot \nabla_x
u^b\cdot\P(|\mathbb{P}R^\nu|^{p-2}\mathbb{P}R^\nu)
\mathrm{d}x,\\[3mm]
\displaystyle B_{17}=\int_\Omega \triangle\varphi\cdot
\partial_zu^b\cdot\P(|\mathbb{P}R^\nu|^{p-2}\mathbb{P}R^\nu)
\mathrm{d}x,
& B_{18}=-\sqrt{\nu}\int_\Omega
v\cdot\nabla_xu^b\cdot\P(|\mathbb{P}R^\nu|^{p-2}\mathbb{P}R^\nu)
\mathrm{d}x,\\[3mm]
\displaystyle B_{19}=\int_\o (\nabla_x q)\cdot (|\P\rr|^{p-2} \P\rr+\nabla Q)\mathrm{d}x.
\end{array}
\end{equation*}

We need to estimate each term in (5.11). We first deal
with the term involving Laplacian:
\begin{equation*}
\begin{aligned}
&-\nu\int_\Omega \triangle
R^\nu\cdot(|\mathbb{P}R^\nu|^{p-2}\mathbb{P}R^\nu+\nabla Q)
\mathrm{d}x\\
=&-\nu\int_\Omega
\triangle(R^\nu)|\mathbb{P}R^\nu|^{p-2}\mathbb{P}R^\nu\mathrm{d}x \\
&-\nu\int_\Omega \triangle(\mathbb{P}R^\nu)\nabla Q\dif x
-\nu\int_\Omega \triangle((I-\mathbb{P})R^\nu)\nabla Q
\mathrm{d}x\\
:=&d_1+d_2+d_3.
\end{aligned}
\end{equation*}

Integration by parts leads to
\begin{equation}\tag{5.12}
\begin{aligned}
d_1=&\nu\int_\Omega (\nabla(R^\nu): \nabla (|\mathbb{P}R^\nu|^{p-2}\mathbb{P}R^\nu)) \mathrm{d}x-\nu \int_{\p\o}(\nabla\rr\cdot \vec{n}) \cdot (|\mathbb{P}R^\nu|^{p-2} \mathbb{P}R^\nu)\dif\sigma\\
\equiv & d_{11}+d_{12}
\end{aligned}
\end{equation}
Note that $|\nabla u|\geq |\nabla|u||$ for any vector $u$. Thus it is easy to derive that
\begin{equation}\tag{5.13}
\begin{aligned}
d_{11}=&\nu\int_\Omega (\nabla(\P R^\nu): \nabla (|\mathbb{P}R^\nu|^{p-2}\mathbb{P}R^\nu))\mathrm{d}x\\
&+ \nu\int_\Omega(\nabla((I-\P)\rr): \nabla(|\P\rr|^{p-2} \P\rr))\mathrm{d}x\\
= & \nu\int_\Omega |\nabla(\P\rr)|^2 |\P\rr|^{p-2} \mathrm{d}x + (p-2)\nu\int_\Omega (|\P\rr|^{\frac{p}{2}-1} |\nabla|\P\rr|)^2 \mathrm{d}x\\
& -\nu\int_\Omega(\nabla((I-\P)\rr): \nabla (|\P\rr|^{p-2} \P\rr))\mathrm{d}x\\
\ge &\frac{1}{2}\nu\int_\o |\nabla(\P\rr)|^2 |\P\rr|^{p-2}\dif x + \frac{(p-2)4}{p^2} \nu \int_\Omega |\nabla| \P\rr|^{\frac{p}{2}}|^2 dx\\
& -C\nu\int_\Omega |\nabla(I-\P)\rr|^2 |\P\rr|^{p-2} dx - C\nu\int_\Omega |\nabla (I-\P)\rr| |\P\rr| |\nabla(|\P\rr|^{p-2})|\mathrm{d}x\\
\ge & \frac{1}{2} \nu\int |\nabla(\P\rr)|^2 |\P\rr|^{p-2} dx + \frac{2(p-2)}{p^2}\nu \int_\Omega |\nabla|\P\rr|^{\frac{p}{2}}|^2 \mathrm{d}x\\
& -C\nu\int_\Omega |\nabla ((I-\P)\rr)|^2 |\P\rr|^{p-2} \mathrm{d}x\\
\ge & \frac{2(p-1)}{p^2} \nu\int|\nabla|\P\rr|^{\frac{p}{2}}|^2 dx - C\|\nabla((I-\P)\rr)\|^p_p - C\|\P\rr\|^p_p.
\end{aligned}
\end{equation}
Next, we handle the boundary term $d_{12}$. Note that due to (5.2), one has
\begin{equation}\tag{5.14}
\begin{aligned}
|d_{12}|=&\left|-\nu\int_{\p\o}(\nabla\rr\cdot\vec{n})\cdot(|\P\rr|^{p-2} \P\rr)\dif \sigma\right|\\
=&\left|-\nu\int_{\p\o} ((\nabla(\rr+v))\cdot\vec{n})\cdot (|\P\rr|^{p-2} \P\rr)\dif \sigma + \nu \int_{\p\o} (\nabla v\cdot\vec{n})\cdot (|\P\rr|^{p-2} \P\rr)\dif \sigma\right|\\
\leq & \nu \left|\int_{\p\o}((\nabla(\rr+v))\cdot\vec{n})\cdot(|\P\rr|^{p-2} \P\rr)\dif\sigma\right| + C\nu ||\nabla v||_{p,\p\o} ||\P\rr||^{p-1}_{p,\p\o}
\end{aligned}
\end{equation}
Since $(|\P\rr|^{p-2} \P\rr)\cdot\vec{n}|_{\p\o}=0$ and
$(\rr+v)\cdot\vec{n}|_{\p\o}=0$ due to (5.2), then it can be
verified that
\begin{equation}\tag{5.15}
\begin{aligned}
& \nu\int_{\p\o} ((\nabla(\rr+v))\cdot\vec{n})\cdot (|\P\rr|^{p-2} \P\rr)\dif\sigma\\
=& \nu\int_{\p\o} (\nabla\times(\rr+v)\times\vec{n})\cdot
(|\P\rr|^{p-2} \P\rr)\dif\sigma - \nu\int_{\p\o}
((D_x(\vec{n})(\rr+v)))\cdot (|\P\rr|^{p-1} \P\rr)\dif\sigma
\end{aligned}
\end{equation}
Note that it follows from the boundary condition (5.3) that
\begin{equation*}
\begin{aligned}
(\nabla\times(\rr+v))\times\vec{n}
=-(\nabla\times b)\times\vec{n} + (\nabla_x \times v)\times\vec{n}.
\end{aligned}
\end{equation*}
This and (5.15) show that
\begin{equation*}
\begin{aligned}
& \left| \nu\int_{\p\o} ((\nabla(\rr+v))\cdot\vec{n})\cdot (|\P\rr|^{p-2} \P\rr)d\sigma\right|\\
=& \bigg|\nu\int_{\p\o} (-(\nabla\times b)\times\vec{n})\cdot (|\P\rr|^{p-2} \P\rr)d\sigma + \nu\int_{\p\o}((\nabla_x\times v)\times\vec{n})\cdot (|\P\rr|^{p-2} \P\rr)d\sigma\\
& -\nu\int_{\p\o} (D_x(\vec{n})(\rr+v)) \cdot(|\P\rr|^{p-2} \P\rr)d\sigma\bigg|\\
\leq & \nu C\left((||b||_{1,p,\p\o} + ||v||_{1,p,\p\o}) ||\P\rr||^{p-1}_{p,\p\o} + ||\P\rr||^p_{p,\p\o} + ||(I-\P)\rr||_{p,\p\o} ||\P\rr||^{p-1}_{p,\p\o}\right)
\end{aligned}
\end{equation*}
This and (5.14) yield
$$|d_{12}|\le \nu C (||b||_{1,p,\p\o} + ||v(t,\cdot,0)||_{1,p,\p\o}) \|\P\rr\|_{p,\p\o}^{p-1} + C\nu\|(I-\P)\rr\|_p \|P\rr\|^{p-1}_{p,\p\o}.$$
Note that $||b||_{1,p,\p\o} \leq \nu^{-\frac{1}{2}} C$ and $ ||v(t,\cdot,0)||_{1,p,\p\o} \leq C$. One can obtain from Lemma 2.1, Lemma 5.3 and Young's inequality that
\begin{equation}\tag{5.16}
|d_{12}|\leq \varepsilon\nu||\nabla|\P\rr|^{\frac{p}{2}}||^2_2 + C||\P\rr||^p_p + C.
\end{equation}
Next, due to the formula $\Delta u=-\nabla\times(\nabla\times u)+\nabla(\di u)$, one has
\begin{equation}\tag{5.17}
\begin{aligned}
d_2+d_3=&-\int_{\o}\Delta(\P\rr)\nabla Q\dif x-\nu\int_\o \Delta((I-\P)\rr)\nabla Q\dif x\\
=&-\int_{\p\o}((\nabla\times(\P\rr))\times \vec{n})\nabla Q\dif x+\nu\int_{\o}\di((I-\P)\rr)\Delta Q\dif x
\end{aligned}
\end{equation}
where one has used the fact that $\nabla Q\cdot n|_{\p\o}=0$. Note that $(\nabla\times(\P\rr))\times\vec{n}=(\nabla\times\rr)\times\vec{n}$. It thus follows from a similar argument for (5.16) that
\begin{equation}\tag{5.18}
\begin{aligned}
&\left|-\int_{\p\o}((\nabla\times(\P\rr))\times\vec{n})\cdot \nabla Q\dif x\right|\\
\leq & C\nu(||b||_{1;p,\p\o} + ||v||_{1,p,\p\o} + ||\P\rr||_{p,\p\o} + ||(I-\P)\rr||_{p,\p\o})||\nabla Q||_{\frac{p}{p-1},\p\o}\\
\leq & C\nu(||b||_{1,p,\p\o} + ||v||_{1,p,\p\o} + ||\P\rr||_{p,\p\o} + ||(I-\P)\rr||_{p,\p\o})||\nabla(|\P\rr|^{\frac{p}{2}})||_{L^2} ||\P\rr||^{\frac{p-2}{2}}_p\\
\leq & C||\P\rr||^p_p + \varepsilon\nu||\nabla(|\P\rr|^{\frac{p}{2}})||^2_2+C,
\end{aligned}
\end{equation}
where one has used (5.7), (5.10), Lemma 2.1 and Lemma 2.3. Furthermore,
\begin{equation}\tag{5.19}
\begin{aligned}
\left|\nu\int_\o\di((I-\P)\rr)\Delta Q\dif x\right| \leq & \nu  |(\di(I-\P)\rr)||_p ||\Delta Q||_{\frac{p}{p-1}}\\
\leq & C||\P\rr||^p_p + \varepsilon\nu||\nabla |\P\rr|^{\frac{p}{2}}||^2_2
\end{aligned}
\end{equation}
where (5.10) and (5.7) have been used. Thus it holds that
\begin{equation}\tag{5.20}
|d_2+d_3| \leq C||\P\rr||^p_p + \varepsilon\nu|| |\P\rr|^{\frac{p}{2}}||^2_2+C.
\end{equation}
Now we turn to the estimates of $B_i$ ($1\leq i\leq 18$).

\textbf{Estimate of $B_1$}: We first rewrite $B_1$ as
\begin{equation*}
\begin{aligned}
B_1=&-\int_\Omega((u^\nu-u^0)\nabla R^\nu)\cdot(|\P\rr|^{p-2} \P\rr+\nabla Q)\dif x\\
&-\int_\Omega (u^0\nabla R^\nu)\cdot (|\P R^\nu|^{p-2}\P R^\nu+\nabla Q)\dif x\\
\equiv & B_{11}+B_{12}.
\end{aligned}
\end{equation*}
Then it follows from integration by parts many times and the
boundary condition $u^0\cdot\vec{n}|_{\p\o} =
\P\rr\cdot\vec{n}|_{\p\o}=0$ that
\begin{equation}\tag{5.21}
\begin{aligned}
|B_{12}|=& \left|-\int_\o u^0\cdot\nabla(\P\rr+(I-\P)\rr)\cdot(||\P\rr||^{p-2} \P\rr+\nabla Q)\dif x\right|\\
=& \bigg|-\int_\o u^0\nabla\left(\frac{1}{p}|\P\rr|^p\right)\dif x - \int_\o(u^0\nabla\P\rr)\nabla Q\dif x\\
& - \int_\o(u^0\nabla(I-\P)\rr)\cdot (|\P\rr|^{p-2} \P\rr+\nabla Q)dx\bigg|\\
\leq & \left| \int_\o(\P\rr\cdot\nabla u^0)\nabla Q\dif x\right| + \left|\int_\o(u^0\nabla(I-\P)\rr)\cdot (|\P\rr|^{p-2} \P\rr+\nabla Q)dx\right|\\
\leq & C||\P\rr||_p ||\nabla Q||_{\frac{p}{p-1}} + C||\nabla(I-\P)\rr||_p \left(||\P\rr||^{p-1}_p + ||\nabla Q||_{\frac{p}{p-1}}\right)\\
\leq & C+||\P\rr||^p_p
\end{aligned}
\end{equation}
where one has used the regularity of $u^0$, (5.10) and Lemma 5.3. Next we estimate $B_{11}$. Rewrite $B_{11}$ as
\begin{equation}\tag{5.22}
\begin{aligned}
|B_{11}|\leq & \left|\int_\Omega ((u^\nu-u^0)\nabla R^\nu) \cdot|\P R^\nu|^{p-2} \P R^\nu\dif x\right| + \left|\int_\Omega
((u^\nu-u^0)\nabla R^\nu)\nabla Q\dif x\right|\\
\equiv & B_{111}+B_{112}.
\end{aligned}
\end{equation}
It follows from (3.1) that
\begin{equation}\tag{5.23}
\begin{aligned}
B_{111}
= & \left|\int_\o((u^\nu-u^0)\nabla(I-\P)\rr) |\P\rr|^{p-2} \P\rr \dif x \right|\\
\leq &(\sqrt{\nu}\|u^b\|_\infty + \nu\|v\|_\infty) \|(I-\P)R^\nu\|_{1,p} \|\P\rr\|_p^{p-1}\\
& +\nu\int_\o|\nabla(I-\P)\rr| \|\P\rr\|^p \dif x + \nu\int_\o |(I-\P)\rr| |\nabla(I-\P)\rr| |\P\rr|^{p-1} \dif x\\
\leq & C+\nu \|\P\rr\|^p_{\frac{p^2}{p-1}} + C\|\P\rr\|^p_p\\
\leq & \nu\|\nabla|\P\rr|^{\frac{p}{2}}\|^{\frac{3}{p}}_2 \|\P\rr\|_p^{\frac{2p-3}{2}} + \|\P\rr\|_p^p +C
\end{aligned}
\end{equation}
Similarly, the term $B_{112}$ can be estimated as follows
\begin{equation*}
\begin{aligned}
B_{112}\leq & \left|\int_\Omega\di((u^\nu-u^0)\otimes\nabla\P\rr)\cdot\nabla Q\dif x\right|\\
&+\left|\int_{\o}(\sqrt{\nu}u^b+\nu v+\nu\rr)\cdot\nabla(I-\P)\rr\cdot\nabla Q\dif x\right|\\
\le & \left| \int_\o(\sqrt{\nu}u^b+\nu v+\nu\rr)\otimes\P\rr:\nabla^2 Q\dif x \right| + (\sqrt{\nu}\|u^b\|_\infty+\nu\|v\|_\infty)\|(I-\mathbb{P})R^\nu\|_{1,p}\|\nabla Q\|_{p'}\\
&+\nu|\P\rr|_{\frac{p^2}{p-1}}\|(I-\mathbb{P})R^\nu\|_{1,p}\|\nabla Q\|_{\frac{p^2}{(p-1)^2}}+\nu\|(I-\P)\rr\|_\infty\|\nabla(I-\P)\rr\|_p\|\nabla Q\|_{p'}\\
\le & C\|\P\rr\|_p^p+\e\nu\|\nabla|\P\rr|^{\frac{p}{2}}\|_2^2+C\nu\|\P\rr\|^{p\frac{p-1}{p-3}}.
\end{aligned}
\end{equation*}

Therefore, we obtain from (5.21)-(5.23) that
\begin{equation}\label{5.9}\tag{5.24}
B_1\le
C\|\P\rr\|_p^{p}+C\nu\|\P\rr\|_p^{p\frac{p-1}{p-3}}+C+\e\nu\|\nabla|\P\rr|^{\frac{p}{2}}\|_2^2.
\end{equation}

\textbf{Estimate of $B_2$+$B_4$}:

Due to the regularity of $u^0$ and the uniform bound for $\p_z u^b$, one can get from the estimate (5.10) that
\begin{equation}\label{5.10}\tag{5.25}
|B_2+B_4|\le \|\P\rr\|_p^p+C.
\end{equation}

\textbf{Estimate of $B_3+B_5$}:

It follows from the construction of $v$ in [12] that
\begin{equation}\tag{5.26}
\begin{aligned}
|B_3+B_5|\le &C\sqrt{\nu}\int_\Omega |R^\nu| |\nabla_x u^b| |\P(|\P\rr|^{p-2} \P\rr)|\dif x\\
\leq & C\sqrt{\nu}\int_\o |\P\rr| |\nabla_x u^b| |\P(|\P\rr|^{p-2} \P\rr)|\dif x\\
& + C\sqrt{\nu}\int_\o|(I-\P)\rr| |\nabla_x u^b| |\P(|\P\rr|^{p-2} \P\rr)|\dif x\\
\equiv & J_1+J_2.
\end{aligned}
\end{equation}
Note that for $p>3$, $\|\nabla_x u^b\|_{2p}\le C\|u^b\|_{1,3,1}$  and
$$ \|(I-\mathbb{P})R^\nu\|_{2p}\leq C\|(I-\mathbb{P})R^\nu\|_{1,p}\le C,$$
due to (5.7) in Lemma 5.3. Thus, one can derive from this and (5.10) that
\begin{equation*}
\begin{aligned}
J_1\le & C\sqrt{\nu}\|\mathbb{P}R^\nu\|_{2p} \|\mathbb{P}R^\nu\|_{p}^{p-1}\\
\leq &
C\sqrt{\nu}\|\nabla|\mathbb{P}R^\nu|^{\frac{p}{2}}\|^{\frac{2}{p}}_{2}\|\mathbb{P}R^\nu\|_{p}^{p-1}\\
\leq &
C\nu^{\frac{p}{2}}\|\nabla|\mathbb{P}R^\nu|^{\frac{p}{2}}\|_2^2+C\|\P\rr\|_p^p\\
\end{aligned}
\end{equation*}
and
\begin{equation*}
\begin{aligned}
J_2 \leq & C\sqrt{\nu}\|(I-\mathbb{P})R^\nu\|_{2p}\|\mathbb{P}R^\nu\|_{p}^{p-1}\\
\le & C\|\P\rr\|_p^p+C\nu^{\frac{p}{2}}.
\end{aligned}
\end{equation*}
Combining this with (5.26) yields that
\begin{equation}\label{5.12}\tag{5.27}
|B_3+B_5|\leq
C\|\mathbb{P}R^\nu\|_p^p+C\nu^{\frac{p}{2}}\|\nabla|\mathbb{P}R^\nu|^{\frac{p}{2}}\|_2^2+C\nu^{\frac{p}{2}}.
\end{equation}

\textbf{Estimate of $B_6$}:

Due to the construction of $v(t,x,z)=-\vec{n}\int^{+\infty}_z \di_x u^b \dif z$, one can rewrite $B_6$ as
\begin{equation}\label{5.13}\tag{5.28}
\begin{aligned}
|B_6|=& \left|\int_\Omega\int_{\frac{\varphi(x)}{\sqrt{\nu}}} \left( \partial_t \di_x u^b(t,x,z) \dif z\right)\vec{n} \cdot \P(|\P\rr|^{p-2} \P\rr)\dif x \right|\\
\leq & \left| \int_\Omega\int_{\frac{\varphi(x)}{\sqrt{\nu}}}^\infty
\di_x \{[u^b\cdot\nabla u^0 + u^0\cdot\nabla_x u^b]_{tan}\} \dif
z\vec{n}\cdot
\P(|\mathbb{P}R^\nu|^{p-2}\P\rr)\dif x\right|\\
&+\left|\int_\Omega\int_{\frac{\varphi(x)}{\sqrt{\nu}}}^\infty
\di_x [fz\cdot\partial_z u^b]\dif z \vec{n}\cdot\P(|\mathbb{P}R^\nu|^{p-2}\P\rr)\dif x\right|\\
&+\left| \int_\Omega\int_{\frac{\varphi(x)}{\sqrt{\nu}}}^\infty \di_x \partial_z^2 u^b\dif z\vec{n}\cdot\P(|\mathbb{P}R^\nu|^{p-2}\P\rr)\dif x\right|\\
\equiv & B_{61} + B_{62} + B_{63},
\end{aligned}
\end{equation}
where one has used (3.6). Due to the regularity estimates of $u^b$ and $u^0$ and Lemma 5.2, one has
\begin{equation}\tag{5.29}
B_{61}\leq C\|\mathbb{P}R^\nu\|_p^p + C\|u^b\|_{1,3,0}^p.
\end{equation}
Similarly, the regularity of $f$ and $\partial_x\partial_zu^b$ yield
\begin{equation}\tag{5.30}
B_{62}\leq C\|\mathbb{P}R^\nu\|_p^p+C\|u^b\|^p_{1,2,1}.
\end{equation}
Finally,
\begin{equation}\tag{5.31}
\begin{aligned}
B_{63}\leq C\|\di_x\partial_z
u^b|_{z=\frac{\varphi(x)}{\sqrt{\nu}}}\|_p
\|\mathbb{P}R^\nu\|_p^{p-1} \le ||\P\rr||^p_p +  C\|u^b\|_{1,1,1}^p.
\end{aligned}
\end{equation}
Consequently, one has
\begin{equation}\label{5.14}\tag{5.32}
\begin{aligned}
|B_6|
&\leq
C\|\mathbb{P}R^\nu\|_p^{p}+C.
\end{aligned}
\end{equation}
Similar analysis yields that
\begin{equation}\label{5.15}\tag{5.33}
|B_7+B_8+B_9|\le C\|\mathbb{P}R^\nu\|_p^{p}+C\sqrt{\nu}\|u^b\|_{1,2,0}^p+\|u^b\|_{1,2,1}^p+C.
\end{equation}

\textbf{Estimate of $B_{10}$}:

By integration by parts, one can rewrite $|B_{10}|$ as
\begin{equation}\label{5.16}\tag{5.34}
\begin{aligned}
|B_{10}|\le &\left| \nu\int_\Omega\nabla_x\left[ v \left( x,\frac{\varphi(x)}{\sqrt{\nu}}\right)\right]: (\nabla(|\mathbb{P}R^\nu|^{p-2} \P\rr) + \nabla^2 Q)\dif x\right|\\
&+\nu \left| \int_{\p\o}\p_n \left( v \left( t,x,\frac{\varphi(x)}{\sqrt{\nu}}\right) \right) \cdot (|\mathbb{P}R^\nu|^{p-2} \P\rr + \nabla Q) \dif\sigma \right|\\
\equiv & B_{101}+B_{102}.
\end{aligned}
\end{equation}
It follows from the regularity estimates of $v$, Lemma 5.1 and its analysis, Young's inequality, and (5.10) that
\begin{equation}\tag{5.35}
\begin{aligned}
|B_{101}|\leq &
C\nu\int_\Omega \left| \nabla_x \left[ v \left( x,\frac{\varphi(x)}{\sqrt{\nu}}\right) \right] \right|^2 |\mathbb{P}R^\nu|^{p-2} \dif x
+ C\nu\int_\Omega |\nabla\mathbb{P}R^\nu|^2 |\mathbb{P}R^\nu|^{p-2} \dif x\\
&+ \nu \left\| \nabla \left[ v \left( x,\frac{\varphi(x)}{\sqrt{\nu}}\right) \right] \right\|_p \|\nabla^2 Q\|_{\frac{p}{p-1}}\\
\leq & C\|\P\rr\|_p^p + \e\nu \|\nabla|\P\rr|^\frac{p}{2}\|_2^2 + C.
\end{aligned}
\end{equation}
Due to the construction of $v(t,x,\frac{\varphi(x)}{\sqrt{\nu}})=\bar{v}\vec{n}$ with $\bar{v}$ being a scalar function given by $\bar{v}(t,x,z)=-\int^\infty_z \di_x u^b(t,x,z)\dif y$, one has $\p_n v=\p_n\bar{v}\vec{n} +\bar{v}\p_n\vec{n}$. Since $\P(|\P\rr|^{p-2} \P\rr)\cdot\vec{n}=0$ on $\p\o$, so
\begin{equation}\tag{5.36}
\begin{aligned}
B_{102}= & \nu\left| \int_{\partial\Omega}(\bar{v}(x,0)\partial_n\vec{n})\cdot (|\mathbb{P}R^\nu|^{p-2}\mathbb{P}R^\nu+\nabla Q)\dif\sigma\right|\\
= & \nu\left|\int_\Omega \di[(\bar{v}(x,0)\partial_n\vec{n}\cdot (|\mathbb{P}R^\nu|^{p-2}\mathbb{P}R^\nu+\nabla Q))\vec{n}]\dif x\right|\\
\leq & \nu \int_\Omega |\nabla[\bar{v}(x,0)\partial_n \vec{n}\cdot\vec{n}]| |(\mathbb{P}R^\nu|^{p-2}\mathbb{P}R^\nu+\nabla Q)|\dif x\\
& + \nu\int_\o |\vec{v}(x,0)\p_n\vec{n}| |\vec{n}| |\nabla(|\P\rr|^{p-2} \P\rr+\nabla Q)|\dif x\\
\leq & C\|\mathbb{P}R^\nu\|_p^p + \e\nu\|\nabla|\P\rr|^{\frac{p}{2}}\|^2_2+C.
\end{aligned}
\end{equation}
Therefore,
\begin{equation} \label{5.17}\tag{5.37}
\begin{aligned}
|B_{10}|\leq
C\|\P\rr\|_p^p + \e\nu\|\nabla|\P\rr|^\frac{p}{2}\|_2^2+C.
\end{aligned}
\end{equation}

\textbf{Estimate of $B_{11}$}:

It follows from the regularity estimates for $u^0$ and $u^b$ and (5.10) that
\begin{equation}\label{5.18}\tag{5.38}
\begin{aligned}
|B_{11}| \leq & \left| \int_\o (u^0+\sqrt{\nu} u^b + \nu v+\nu R^\nu) \cdot \nabla_x v\cdot (|\P\rr|^{p-2}\P\rr+\nabla Q) \dif x\right|\\
\le & C(\|u^0\|_{\infty}\|\nabla_x v\|_p + \sqrt{\nu} \|u^b\|_{2p} \|\nabla_x v\|_{2p} + \nu \|v\|_{2p} \|\nabla_x v\|_{2p}) (\|\mathbb{P}R^\nu\|_p^{p-1} + \|\nabla Q\|_{\frac{p}{p-1}})\\
& +\nu\|\rr\|_{2p} \|\nabla_x v\|_{2p} (\|\P\rr\|_p^{p-1} + \|\nabla Q\|_{\frac{p}{p-1}})\\
\le & C (\|u^0\|_{\infty} + \|u^b\|_{1,3,0} + \sqrt{\nu}\|u^b\|_{1,3,1} + \|u^b\|_{1,2,0,2p})^p\\
&+\nu\|\P\rr\|_{2p} \|u^b\|_{1,2,0,2p} \|\P\rr\|_p^{p-1} + C\|\P\rr\|_p^p\\
&\leq \e\nu \|\nabla|\P\rr|^{\frac{p}{2}}\|_2^2 + C\|\P\rr\|_p^p + C_2.
\end{aligned}
\end{equation}

\textbf{Estimate of $B_{12}+\sum_{i=14}^{18}B_{i}$}:

Applying similar analysis and using the bounds on $u^b$, $\p_z u^b$, $u^0$ and $v$, we can get
\begin{equation}\label{5.19}\tag{5.39}
\left| B_{12}+\sum_{i=14}^{18}B_{i}\right| \leq C\|\P\rr\|_p^p+C.
\end{equation}

\textbf{Estimate of $B_{13}$}:

It follows from the definition of $v$, (4.1) and (5.10) that
\begin{equation*}
\begin{aligned}
|B_{13}|= & \left| -\frac{1}{\nu} \int_\Omega (u^0\cdot \vec{n}) \di_x u^b\vec{n} \cdot (|\P\rr|^{p-2} \P\rr + \nabla Q)\dif x\right|\\
= & \left|\int_\o \frac{u^0\cdot\vec{n}}{\varphi} (z \di_x u^b)\bigg|_{z=\frac{\varphi}{\sqrt{\nu}}} \vec{n} \cdot (|\P\rr|^{p-2} \P\rr + \nabla Q) \dif x \right|\\
\leq & C \|(1+z)\di_x u^b|_{z=\frac{\varphi(x)}{\sqrt{\nu}}} \|_p \cdot \|\mathbb{P}R^\nu\|_p^{p-1}\\
\leq & C \nu^{\frac{1}{2p}} \|u^b\|_{1,2,0,p} \|\P\rr\|^{p-1}_p\\
\leq & C \|\P\rr\|^p_p + C\nu^{\frac{1}{2}}.
\end{aligned}
\end{equation*}

\textbf{Estimate of $B_{19}$}:

Finally, we estimate $B_{19}$. Since
\begin{equation*}
\nabla_x q=-\int_z^\infty((\nabla_xu^b(z))\cdot\nabla
u^0+u^0\cdot(\nabla_x^2 u^b(z)))\cdot \vec{n}+(u^b\cdot\nabla
u^0+u^0\cdot\nabla_xu^b)\cdot\nabla_x \vec{n} \mathrm{d}z.
\end{equation*}
Then
\begin{equation*}
\|\nabla_x q\|_p\leq C\|u^b\|_{L^1_z(R_+,W^{2,p}(\Omega))}\leq
C\|u^b\|_{1,3,1}.
\end{equation*}
Therefore
\begin{equation}\label{5.21}\tag{5.21}
\begin{aligned}
B_{19} \leq C\|\mathbb{P}R^\nu\|_p^{p}+C_6.
\end{aligned}\end{equation}

 Collecting (5.12), (5.13), (5.16), (5.20) and all the estimates on $B_i(i=1,\cdots,19)$, one deduces from
(5.11) that, for suitably small $\e$, it holds that
\begin{equation}\label{5.22}\tag{5.22}
\begin{aligned}
\frac{d}{dt}\|\mathbb{P}R^\nu\|_p^{p}+c_0\nu\|\nabla|\P\rr|^{\frac{p}{2}}\|_2^2\leq
C\|\mathbb{P}R^\nu\|_p^{p}+C+\nu\|\P\rr\|_p^{p\frac{p-1}{p-3}}.
\end{aligned}
\end{equation}
Since $\frac{p-1}{p-3}>1$, so part (c) of Gronwall's Lemma implies
that there exists a small $0<\nu_0< 1$ such that for all
$0<\nu\le\nu_0$, the desired estimate (5.8) in Lemma 5.4 holds. This
completes the proof of Lemma 5.4.

\section{$H^1$-estimates of the remainder $R^\nu$}

In this section, we will derive the $H^1$-estimates for the
remainder $R^\nu$. To this end, we need to apply Lemma 2.4 to handle
the boundary terms in the $L^2$-estimate of the vorticity. However,
it is noted that $R^\nu$ does not satisfy the conditions in Lemma
2.4. To overcome this difficulty, we set
\begin{equation}
R(t,x)=R^\nu(t,x)+b(t,x)
\end{equation}
where $b(t,x)$ is defined in Section 5. It then follows from (5.1)-(5.4) that $R(t,x)$ solves the following system
\begin{equation}
\begin{aligned}
&\partial_tR-\nu \triangle R+u^\nu\cdot\nabla
R+R\cdot\nabla u^0+\sqrt{\nu}R\cdot \vec{n}\partial_z
v+R\cdot \vec{n}\partial_zu^b+\sqrt{\nu}R\cdot
\nabla_xu^b\\
=&R.H.S. +\p_t b-\nu\triangle b+u^\nu\cdot\nabla b+b\cdot\nabla u^0 +\sqrt{\nu} b\cdot\vec{n}\p_z v+b\cdot\vec{n}\p_z u^b+\sqrt{\nu} b\cdot\nabla_x u^b,
\end{aligned}
\end{equation}
$$\mathrm{div}R(t,x)=-
 \mathrm{div}_x v\left( t,x,\frac{\varphi(x)}{\sqrt{\nu}}\right) +\mathrm{div} b(t,x),$$
with boundary and initial conditions as
\begin{alignat}{12}
& R\cdot \vec{n}=0, \qquad (\C R)\times\vec{n}=0, \qquad \mbox{on}~~\partial\Omega,\tag{6.3}\\
& R(0,x)=0, \qquad \mbox{in}~~\Omega.\tag{6.4}
\end{alignat}
Here R.H.S. is defined in (5.1).

It follows from the analysis in Section 5 and in [12] that
\begin{alignat}{12}
& ||R^\nu||_{L^\infty(0,T; L^p(\o))} \leq C, \qquad 3<p\leq 6,\tag{6.5}\\
& ||b||_{L^\infty(0,T; H^1)} \leq C\nu^{-\frac{1}{2}},\tag{6.6}\\
& ||R||_{L^\infty(0,T; L^2(\o))} + ||\mathrm{div}\,R||_{L^\infty(0,T; L^2(\o))} \leq C\nu^{-\frac{1}{2}}.\tag{6.7}
\end{alignat}
It follows from (6.3), (6.5)-(6.8), and Lemma 2.2 that
\begin{equation}\tag{6.8}
||R||_{1,2} \leq ||\nabla\times R||_2 + C\nu^{-\frac{1}{2}}
\end{equation}
Therefore, it suffices to estimate $||\nabla\times R||_{L^\infty(0,T; L^2(\o))}$. Set $w=\nabla\times R$. Then (6.2) implies that
\begin{equation}\tag{6.9}\begin{aligned}
&\partial_t\om-\nu \Delta \om+\C(u^\nu\cdot\nabla
R)+\C(R\cdot\nabla u^0)+\\
&\sqrt{\nu}\C(R\cdot \vec{n}\partial_z
v)+\C(R\cdot \vec{n}\partial_zu^b)+\sqrt{\nu}\C(R\cdot
\nabla_xu^b)\\
=&\C R.H.S+\C\p_tb-\nu\Delta\C b+\C(u^\nu\cdot\nabla b)+\C(b\cdot\nabla u^0)\\
&+\sqrt{\nu}\C(b\cdot \vec{n}\p_z v)+\C(b\cdot \vec{n}\p_z u^b)+\sqrt{\nu} \C(b\cdot\nabla u^b).
\end{aligned}
\end{equation}
Multiply (6.9) by $\om$ and integrate on $\o$ to get
\begin{equation}\tag{6.10}
\frac{\dif}{\dif t}\|\om\|^2-\nu\int_{\o}\Delta\om\cdot\om\dif x=\int_{\o} \C (R.H.S.)\cdot\om\dif x+\sum^{12}_{i=1}E_i.
\end{equation}
First, we estimate the second term on the left hand side. Integration by parts yields
$$-\nu\int_{\o}\Delta\om\cdot\om\dif x=+\nu\int_{\o}|\nabla\om|^2\dif x-\nu\int_{\p\o}(\vec{n}\cdot\nabla\om)\cdot\om\dif\sigma.$$
It follows from (6.3) and Lemma 2.4 that
$$\nu\left|\int_{\p\o}(\vec{n}\cdot\nabla\om)\cdot\om\dif\sigma\right|\leq C\nu\int_{\p\o} |\om|^2\dif\sigma +C\nu \left|\int_{\p\o} \sum^3_{n=1} [\om\times\nabla[\om\times\vec{n}]_n]_n\dif\sigma\right|.$$
Since $\om\times\vec{n}|_{\p\o}=0$, so $\nabla[\om\times\vec{n}]_n$ is parallel to $\vec{n}$, thus $[\om\times\nabla[\om\times\vec{n}]_n]_n|_{\p\o}=0$. Hence the last term above is zero. It follows from this, Lemma 2.1 and Lemma 2.3 that
$$\nu\left|\int_{\p\o}(\vec{n}\cdot\nabla\om)\cdot\om\dif\sigma\right|\leq \e\nu\int_\o|\nabla\om|^2 \dif x+C(\e)\nu\int_\o|\om|^2\dif x.$$
Consequently, one gets
\begin{equation}\tag{6.11}
-\nu\int_\o\Delta\om\cdot\om\dif x \geq(1-\e)\nu\int_\o|\nabla\om|^2 \dif x - C\nu\int_\o|\om|^2\dif x.
\end{equation}
Next, we estimate the terms on the right hand side of (6.10).

First, due to (6.2), (6.7), Lemma 2.3, and Theorem 2.8, one can get
\begin{equation}\tag{6.12}
\begin{aligned}
E_1=&\int_\o(\nabla\times(u^\nu\cdot\nabla R))\cdot\om\dif x\leq\int_\o|\nabla u^\nu|\,|\nabla R|\,|\om|\dif x+\left|\int_\o (u^\nu\cdot\nabla \om)\cdot\om\dif x\right|\\
=&\int_{\o}|\nabla u^\nu|\,|\nabla R|\,|\om|\dif x\leq ||\nabla u^\nu||_{\infty} ||\nabla R||_2\,||\om||_2\\
\leq & C\|\nabla u^\nu\|_{\infty} (\|\om\|_2 +\|\mathrm{div}\,R\|_2) \|\om\|_2\\
\leq & C\|\nabla u^\nu\|_{\infty} \|\om\|_2^2 + C\nu^{-1}.
\end{aligned}
\end{equation}
Due to Lemma 2.3, (6.5) and (6.6), one has
\begin{equation}\tag{6.13}
\begin{aligned}
|E_2|=&\left|\int_\o \nabla\times(R\cdot\nabla u^0)\cdot \om\dif x\right| \le \int_\o|\nabla R|\,|\nabla u^0|\,|\om|\dif x+ \left|\int_\o (R\cdot\nabla \C u^0)\cdot \om\dif x\right|\\
\le & \|\nabla u^0\|_{\infty} \|\nabla R\|_2\,\|\om\|_2+C\|R\|_6\,\|\nabla \C u^0\|_3\,\|\om\|_2\\
\le & C \|\om\|_2^2+ C\nu^{-1}.
\end{aligned}
\end{equation}
Next, note that
\begin{equation*}
\begin{aligned}
E_3&=\sqrt{\nu}\int_\o\nabla\times(R\cdot \vec{n}\p_z v)\cdot\om\dif x\\
=&\sqrt{\nu}\left(\int_\o (R\cdot \vec{n})(\nabla\times\p_z v)\cdot\om\dif x+\int_\o \p_z v\times(\nabla(R\cdot \vec{n}))\cdot\om \dif x\right)\\
\equiv& K_1+K_2.
\end{aligned}
\end{equation*}
One has by Lemma 2.3 and (6.7) that
\begin{equation*}
\begin{aligned}
|K_2|=&\sqrt{\nu}\left| \int_\o \nabla(R\cdot \vec{n})\times\p_z v\cdot\om \dif x\right|\\
\le & \sqrt{\nu} \|\nabla (R\cdot \vec{n})\|_2\,\|\p_z v\|_6\,\|\om\|_3 \le C\sqrt{\nu}\|\om\|^\frac{1}{2}_2\|\nabla\om\|^{\frac{1}{2}}_2 \|\nabla (R\cdot \vec{n})\|_{2}\,\|u^b\|_{1,2,1}\\
\le & \frac{1}{2}\nu\e\|\nabla\om\|^2_2 + C\nu^{\frac{1}{3}} \|\om\|^{\frac{2}{3}}_2 (\|\om\|^{\frac{2}{3}}_2+\|\mathrm{div}\,R\|^{\frac{2}{3}}_2)\\
\leq & \frac{1}{2} \nu\e\|\nabla\om\|^2_2 + C\|\om\|^2_2+C.
\end{aligned}
\end{equation*}
Noting that $\nabla_x\varphi=\vec{n}$, one can get
\begin{equation*}
\begin{aligned}
|K_1|=&\sqrt{\nu} \left|\int_\o R\cdot \vec{n} (\nabla\times\p_z v)\cdot\om\dif x\right|\\
=& \sqrt{\nu} \left|\int_\o (R\cdot \vec{n})\cdot(\nabla\times(\p_z \bar{v}\vec{n}))\cdot\om \dif x\right|\\
=& \sqrt{\nu} \left|\int_\o (R\cdot \vec{n})(\p_z\bar{v}(\nabla\times\vec{n})+\vec{n}\times\nabla(\p_z\bar{v}))\cdot\om\dif x\right|\\
=& \sqrt{\nu} \left|\int_\o (R\cdot \vec{n})((\p_z\bar{v})(\nabla\times\vec{n})+\vec{n}\times(\nabla_x\p_z\bar{v} + \frac{1}{\sqrt{\nu}} (\p^2_z\bar{v})\vec{n}))\cdot\om \dif x\right|\\
\le & \sqrt{\nu} \left| \int_\o (R\cdot\vec{n})((\p_z\bar{v})\nabla\times\vec{n})\cdot\om\dif x\right| + \sqrt{\nu} \left| \int_\o (R\cdot\vec{n})(\vec{n}\times\nabla_x\partial_z\bar{v})\om\dif x\right|\\
\le & \sqrt{\nu}\|R\|_6 (\|\p_z\bar{v}(\nabla\times\vec{n})\|_2 + \|\vec{n}\times(\nabla_x\partial_z\bar{v})\|_2) \|\om\|_3\\
\leq & C\sqrt{\nu}\|\nabla R\|_2 (\|(\mathrm{div}_x u^b)(\nabla\times\vec{n})\|_2 + \|\vec{n}\times(\nabla_x\mathrm{div}_x u^b)\|_2) \|\om\|^{\frac{1}{2}}_2 \, \|\nabla\om\|_2^\frac{1}{2}\\
\le & \frac{1}{2}\e\nu\|\nabla\om\|_2^2 + C\|\om\|^2_2 + C\nu^{-\frac{1}{2}}.
\end{aligned}
\end{equation*}
Hence, we obtain
\begin{equation}\tag{6.14}
E_3\le C\nu^{-\frac{1}{2}}+C\|\om\|_2^2+\e\nu\|\nabla\om\|_2^2.
\end{equation}
To estimate $E_4$, we note that
\begin{equation*}
\begin{aligned}
E_4\equiv &\int_\o \nabla \times(R\cdot \vec{n}\p_z u^b)\cdot\om\dif x\\
= & \int_\o (\p_z u^b\times\nabla(R\cdot\vec{n}))\cdot\om\dif x +
\int_\o (R\cdot \vec{n})(\nabla_x\times\p_z u^b)\cdot\om \dif
x+\frac{1}{\sqrt{\nu}}\int_\o (R\cdot \vec{n})(\vec{n}\times\p_z^2
u^b)\cdot\om dx.
\end{aligned}
\end{equation*}
Since
$R\cdot\vec{n}=R^\nu\cdot\vec{n}+b\cdot\vec{n}=R^\nu\cdot\vec{n}+\nu(t,x,0)\cdot\vec{n}$
due to (3.8), thus
$$||R\cdot\vec{n}||_{L^\infty(0,T; L^6(\o))} \leq C,$$
by (6.5) and the regularity estimates for $u^b$. It follows that
\begin{equation}\tag{6.15}
\begin{aligned}
|E_4|\le &\|\p_z u^b\|_{\infty}\|\nabla(R\cdot \vec{n})\|_2\|\om\|_2+\|R\cdot\vec{n}\|_6\|\nabla_x\times(\p_z u^b) \|_3\|\om\|_2\\
&+\frac{1}{\sqrt{\nu}} \|R\cdot \vec{n}\|_{6} \|\p_z^2 u^b\|_2 \|\om\|_3\\
\le & C \|\om\|_2^2+C \nu^{-1}+ \frac{C}{\sqrt{\nu}} \nu^{\frac{1}{4}} \|u^b\|_{0,1,2} \|\om\|^{\frac{1}{2}}_2 \|\nabla\om\|^{\frac{1}{2}}_2\\
\le & C\|\om\|_2^2+C \nu^{-1}+\e\nu\|\nabla\om\|^2_2.
\end{aligned}
\end{equation}
Integrating by parts and using (6.3), one can get
\begin{equation*}
\begin{aligned}
|E_5|&=\left|-\sqrt{\nu}\int_\o R\cdot\nabla_x u^b\C\om\dif x\right|\le \sqrt{\nu}\|R\|_6\|\nabla_x u^b\|_3\|\C\om\|_2\\
&\le \sqrt{\nu}\|\nabla_x u^c\|_{1,2,1}\|\nabla R\|_2\|\C\om\|_2\le C \nu^{-1}+\|\om\|_2^2+\e\nu\|\nabla\om\|_2^2.
\end{aligned}
\end{equation*}
Since $\p_t b(x,t)=\p_tv(t,x,0)+\frac{1}{\sqrt{\nu}}\p_t u^b(t,x,0)$, hence
\begin{equation*}
\begin{aligned}
\|\p_t \C b(x,t)\|_2&=\|\p_t\C_xv(t,x,0)+\frac{1}{\sqrt{\nu}}\p_t \C_x u^b(t,x,0)\|\\
&\le\|\p_t\C_xv(t,x,0)\|_2+\frac{1}{\sqrt{\nu}}\|\p_t \C_x u^b(t,x,0)\|\\
&\le C\|\p_t u^b\|_{1,2,0}+\frac{1}{\sqrt{\nu}}\|\p_tu^b\|_{1,1,1}
\end{aligned}
\end{equation*}
Therefore, $E_6$ can be bounded as follows,
\begin{equation}\tag{6.16}
\begin{aligned}
|E_6|\le \|\p_t \C b(x,t)\|_2\|\om\|_2\le C(\|\p_t
u^b\|^2_{1,2,0}+\nu^{-1}\|\p_tu^b\|^2_{1,1,1})+\|\om\|_2^2.
\end{aligned}
\end{equation}
Integrating by parts shows that $E_7=\nu\int_\o\Delta_x b\cdot \C\om\dif x$.  Hence
\begin{equation}\tag{6.17}
\begin{aligned}
|E_7|&\le\nu\|\Delta b(x,t)\|_2\|\C\om\|_2\le\nu(\|\Delta_x v(t,x,0)\|_2+\frac{1}{\sqrt{\nu}}\|\Delta_xu^b(t,x,0)\|_2)\|\C\om\|_2\\
&\le C(\nu\|u^b\|^2_{1,3,0}+\|u^b\|^2_{1,2,1}) +\e\nu\|\nabla\om\|_2^2.
\end{aligned}
\end{equation}

Since $\|b\|_\infty+\|b\|_{H^2}+ ||\nabla_x b||_\infty + ||\nabla^2_x b||_\infty \le C\|u^b\|_{1,3,1}\nu^{-\frac{1}{2}}$, so
\begin{equation}\tag{6.18}
\begin{aligned}
|E_8|=&\left|\int_\o\nabla\times(u^\nu\cdot\nabla b)\cdot\om\dif x\right|\\
\leq &\int_\o|\nabla u^\nu| \,|\nabla b|\,|\om|\dif x + \left| \int_\o u^\nu\cdot\nabla(\nabla\times b)\cdot \om\dif x\right|\\
\leq & \|\nabla u^\nu\|_\infty \|\nabla b\|_2 \|\om\|_2+\|\nabla^2_x b\|_\infty \|u^\nu\|_2 \|\om\|_2\\
\leq & C\|\om\|^2_2+C\nu^{-1}.
\end{aligned}
\end{equation}
Similarly,
\begin{equation}\tag{6.19}
\begin{aligned}
|E_9|=&\left|\int_\o\nabla\times(b\times\nabla u^0)\cdot\om\dif x\right|\\
\leq & \left|\int_\o |(\nabla\times b)| |\nabla u^0| |\om|\dif x\right| + \int_\o |b| |\nabla\cdot(\nabla\times u^0)| |\om|\dif x\\
\leq & C\|\om\|^2_2+ C\nu^{-1}.
\end{aligned}
\end{equation}
Due to (3.8), it holds that $b\cdot\vec{n}=\frac{1}{\sqrt{\nu}}u^b(t,x,0)\cdot\vec{n}+v(t,x,0)\vec{n}=v(t,x,0)\vec{n}=\bar{v}(t,x,0)=-\int^\infty_0\mathrm{div}_x u^b(t,x,z)\dif z$, one can get
\begin{equation}\tag{6.20}
\begin{aligned}
|E_{10}+E_{11}+E_{12}| \le & \sqrt{\nu}\left|\int_\o \bar{v}(t,x,0) \di_x u^b\vec{n}\cdot\nabla\times\om\dif x \right| + \left| \int_\o \bar{v}(t,x,0)\p_z u^b\cdot(\nabla\times\om)\dif x\right|\\
&+\sqrt{\nu} \left|\int_\o (b\cdot\nabla u^b)\cdot\C\om\dif x\right|\\
\le & \e\nu\|\nabla\om\|_2^2+C\nu^{-1}.
\end{aligned}
\end{equation}
It remains to estimate the term $\int_\o\C (R.H.S)\cdot\om\dif x.$ Set
\begin{equation*}
\int_\o\C (R.H.S)\cdot\om\dif x=\sum_{i=1}^{14} D_i,
\end{equation*}
Where
\begin{equation*}
\begin{array}{ll}
\displaystyle D_1=-\int_\Omega \C(\partial_tv)\cdot\om \mathrm{d}x,
& \displaystyle D_2=\int_\Omega \C\triangle u^0 \cdot\om \mathrm{d}x,\\[3mm]
\displaystyle D_3=\sqrt{\nu}\int_\Omega \C[\triangle_xu^b]\cdot\om\mathrm{d}x,
& \displaystyle D_4=\int_\Omega \C(2n\cdot\nabla_x\partial_zu^b)\cdot\om \mathrm{d}x, \\[3mm]
\displaystyle D_5=\nu\int_\Omega \C(\triangle_x[v(x,\frac{\varphi(x)}{\sqrt{\nu}}])\cdot\om
\mathrm{d}x, & \displaystyle D_6=-\int_\Omega \C(u^\nu\cdot\nabla_xv)\cdot\om\mathrm{d}x,\\[3mm]
\displaystyle D_7=-\int_\Omega \C(v\cdot\nabla u^0)\cdot\om
\mathrm{d}x, & \displaystyle D_8=-\frac{1}{\sqrt{\nu}}\int_\Omega
\C( u^0\cdot
n\partial_zv)\cdot\om\mathrm{d}x,\\[3mm]
\displaystyle D_9=-\sqrt{\nu}\int_\Omega \C(v\cdot
\vec{n}\partial_z v)\cdot\om
\mathrm{d}x, & \displaystyle D_{10}=-\int_\Omega \C(v\cdot
\vec{n}\partial_zu^b)\cdot\om
\mathrm{d}x,\\[3mm]
\displaystyle D_{11}=-\int_\Omega \C(u^b\cdot \nabla_x
u^b)\cdot\om \mathrm{d}x, & \displaystyle D_{12}=\int_\Omega \C(\triangle\varphi\cdot
\partial_zu^b)\cdot\om\mathrm{d}x,\\[3mm]
\displaystyle D_{13}=-\sqrt{\nu}\int_\Omega
\C(v\cdot\nabla_xu^b)\cdot\om
\mathrm{d}x,
& \displaystyle D_{14}=\frac{1}{\sqrt{\nu}}\int_\Omega \C (\nabla_x q)\cdot\om\mathrm{d}x.
\end{array}
\end{equation*}

As in the argument in the estimate of $E_6$, one can infer
\begin{equation}\tag{6.21}
\begin{aligned}
|D_1|&\le C(\|\p_t u^c\|^2_{1,2,0})\nu^{-1}+C\|\om\|_2^2.
\end{aligned}
\end{equation}
It follows from the regularity of $u^0$ and $u^b$ and integration by parts due to (6.3) that
\begin{equation}\tag{6.22}
\begin{aligned}
|D_2|&=\left|\int_\o\C\Delta u^0\cdot\om\right|\le C+\|\om\|_2^2,\\
|D_3|&=\sqrt{\nu}\left|\int_\o\Delta_xu^b\C\om\dif x\right|\le \sqrt{\nu}\|\Delta_xu^b\|_2\|\C\om\|_2\\
&\le C\|u^b\|_{0,2,0}^2+\e\nu\|\C\om\|_2^2.
\end{aligned}
\end{equation}
Since $\C_x[2\vec{n}\cdot\nabla_x\p_z u^b]=2\C \vec{n}\cdot\nabla_x\p_z u^b-2\vec{n}\cdot\nabla_x\C_x u^b-2\vec{n}\cdot\nabla_x(\p_z^2 u^b\times \vec{n})\frac{1}{\sqrt{\nu}}$, thus
\begin{equation}\tag{6.23}
\begin{aligned}
|D_4|\le C\|u^b\|_{0,1,2}^2\nu^{-1}+C\|u^b\|_{0,2,1}^2+\|\om\|_2^2.
\end{aligned}
\end{equation}
Due to
\begin{equation*}
\begin{aligned}
\Delta_x\left[h(x,\frac{\varphi(x)}{\sqrt{\nu}})\right]&=\Delta_x h(x,\frac{\varphi(x)}{\sqrt{\nu}})+2\frac{n(x)}{\sqrt{\nu}}\cdot\nabla_x\p_z h(x,\frac{\varphi(x)}{\sqrt{\nu}})+\frac{\Delta\varphi}{\sqrt{\nu}}\p_z h(x,\frac{\varphi(x)}{\sqrt{\nu}})\\
&+\frac{1}{\nu}\p_z^2 h(x,\frac{\varphi(x)}{\sqrt{\nu}}),
\end{aligned}
\end{equation*}
we can estimate the $D_5$ to obtain
\begin{equation}\tag{6.24}
\begin{aligned}
|D_5|\le & \nu\left|\int_\o \Delta_x\left[v(t,x,\frac{\varphi}{\sqrt{\nu}})\right]\C\om\dif x\right|\\
\le & \nu\left|\int_\o \left(\Delta_x v(t,x,\frac{\varphi}{\sqrt{\nu}})\right.\right.+2\frac{\vec{n}(x)}{\sqrt{\nu}}\cdot\nabla_x\p_z v(t,x,\frac{\varphi(x)}{\sqrt{\nu}})+\frac{\Delta\varphi}{\sqrt{\nu}}\p_z v(t,x,\frac{\varphi(x)}{\sqrt{\nu}})\\
&\left.\left.+\frac{1}{\nu}\p_z^2 v(t,x,\frac{\varphi(x)}{\sqrt{\nu}}\right)\C\om\dif x\right|\\
\le & C\|u^b\|^2_{1,3,0}+\frac{1}{2}\e\nu\|\nabla\om\|_2^2+\left|\int_\o \vec{n}\di_x \p_zu^b\cdot\C\omega\dif x\right|\\
\le & C(\|u^b\|^2_{1,3,0}+\|u^b\|^2_{0,1,1}\nu^{-1})+\e\nu\|\nabla\om\|_2^2.
\end{aligned}
\end{equation}
Next, one has
\begin{equation}\tag{6.25}
\begin{aligned}
|D_6|=& \left|\int_\o u^\nu\cdot\nabla_x v\cdot\om\dif x\right|\\
\le & ||\nabla\om||_2 \, ||\nabla_x v||_3 \, ||u^\nu||_6\\
\le & C\|\nabla\om\|_2 \|u^b\|_{1,3,1}\|\nabla u^\nu\|_6\\
\le & C\|\nabla\om\|_2 \|u^b\|_{1,3,1}\\
\le & \e\nu\|\nabla\om\|^2_2 + C\nu^{-1} \|u^b\|^2_{1,3,1}.
\end{aligned}
\end{equation}
Direct estimate using (4.1) leads to
\begin{equation}\tag{6.26}
|D_7|=\left|\int_\o(v\cdot\nabla u^0\cdot\C\om)\dif x\right| \le \e\nu\int_\o |\nabla\om|^2 \dif x + O(1)\nu^{-1}.
\end{equation}
Next, it follows from the regularities of $f$ and integration by parts that
\begin{equation}\tag{6.27}
\begin{aligned}
|D_8|= & \left|\int_\o\frac{ u^0\cdot n}{\varphi}(z\p_z v)\big|_{z=\frac{\varphi(x)}{\sqrt{\nu}}}\cdot\C\om\right|\\
\le & \|f\|_\infty \|u^b\|_{2,1,0} \|\nabla\times\om\|_2 \le \e\nu\|\nabla\om\|_2^2+C\nu^{-1}.
\end{aligned}
\end{equation}
Similarly, it follows from the uniform bounds on $\p_z u^b$ and $\Delta\phi$ that
\begin{equation}\tag{6.28}
\left|\sum_{i=9}^{13}D_i\right| \le C+C\nu^{-1} + \e\nu\|\nabla\om\|_2^2,
\end{equation}
where $C$ depends on $\|u^b\|_{L^\infty(0,T;H_{1,3,1})}$.

It remains to estimate $D_{14}$. To this end, we note that $\C[\nabla_x q]=\p_z\nabla_x q\times \vec{n}$. Thus, due to (3.7), one has
\begin{equation}\tag{6.29}
\begin{aligned}
|D_{14}|= & \frac{1}{\sqrt{\nu}}\left|\int_\o (\p_z\nabla_x q\times \vec{n})\cdot\om\dif x\right|\le \frac{1}{\sqrt{\nu}} \|\nabla_x\p_z q\|_2\|\om\|_2\\
\le & \frac{1}{\sqrt{\nu}}\|\nabla\left( (u^0\cdot\nabla_x u^b+u^b\cdot\nabla u^0)\cdot \vec{n}\right)\|_2\|\om\|_2\\
\le & C\nu^{-\frac{1}{2}}\|u^b\|^2_{0,2,0}+\|\om\|^2_2.
\end{aligned}
\end{equation}
As a consequence of all the estimates (6.11)-(6.29) and Propositions 4.3-4.6, we obtain
\begin{equation}\tag{6.30}
\frac{\dif}{\dif t}\|\om\|^2_2+\frac{\nu}{2} \int_\o |\nabla\om|^2\dif x \le C\|\om\|_2^2+C(1+\|\p_t u^b\|_{1,1,1}^2)\nu^{-1}.
\end{equation}
Since $(\|\p_t u^b\|_{1,1,1}^2+1)$ is integrable on [0,T], so (6.30) implies that $||\om||_{L^\infty(0,T; L^2(\o))}\leq C\nu^{-\frac{1}{2}}$. This and (6.8) show that $\sup_{0\leq t\leq T}||R||_{1,2}\leq C\nu^{-\frac{1}{2}}$. Hence
\begin{equation}\tag{6.31}
\sup_{0\leq t\leq T} \|R^\nu\|_{1,2} \le C\nu^{-\frac{1}{2}}.
\end{equation}
Note that, due to Lemma 4.1, one has that
\begin{equation}\tag{6.32}
\left\|\nabla_x[v(t,x,\frac{\varphi(x)}{\sqrt{\nu}})]\right\|_2 \leq \|u^b\|_{1,2,0} + \nu^{-\frac{1}{4}} \|u^b\|_{0,2,0}.
\end{equation}
Thus, we have shown that for $\nu\in(0,\nu_0]$,
\begin{equation}\tag{6.33}
\begin{aligned}
& \sup_{0\leq t\leq T} \|u^\nu-u^0\|_{1,2}\leq C\nu^{-\frac{1}{4}},\\
& \int^T_0 \|u^\nu-u^0-\sqrt{\nu} u^b\|^2_{2,2}\dif t\leq C.
\end{aligned}
\end{equation}
Furthermore, similar to the derivation of (6.30), one can show that for $p\in(3,6]$, it holds that
\begin{equation}\tag{6.34}
\frac{\dif}{\dif t} ||\om||^p_p\leq C(||u^b||^{p'}_{1,3,3}+1) ||\om||^p_p +C\nu^{-\frac{p}{2}},
\end{equation}
since $p>3$, so it follows from the Sobolev's embedding that
\begin{equation}\tag{6.35}
||R^\nu||_{L^\infty((0,T)\times \o)} \leq C\nu^{-\frac{1}{2}}.
\end{equation}
Consequently,
\begin{equation}\tag{6.36}
||u^\nu-u^0||_{L^\infty((0,T)\times \o)}\leq C\nu^{\frac{1}{2}}.
\end{equation}
Thus, the proof of Theorem 3.4 is completed.

\end{document}